\documentclass[reqno]{amsart}
\usepackage{amssymb,amsmath,amsthm,mathtools}
\usepackage[displaymath,mathlines]{lineno}
\usepackage{ifpdf}
\ifpdf
 \usepackage[hyperindex]{hyperref}
\else
 \expandafter\ifx\csname dvipdfm\endcsname\relax
 \usepackage[hypertex,hyperindex]{hyperref}
 \else
 \usepackage[dvipdfm,hyperindex]{hyperref}
 \fi
\fi
\numberwithin{equation}{section}
\newtheorem{thm}{Theorem}[section]

\theoremstyle{definition}

\theoremstyle{remark}
\newtheorem{rem}{Remark}[section]

\allowdisplaybreaks[4]
\DeclareMathOperator{\td}{d}

\begin{document}

\title[Refinements of Young's integral inequality]
{Refinements of Young's integral inequality via fundamental inequalities and mean value theorems for derivatives}

\author[F. Qi]{Feng Qi}
\address[Qi]{College of Mathematics, Inner Mongolia University for Nationalities, Tongliao 028043, Inner Mongolia, China; Institute of Mathematics, Henan Polytechnic University, Jiaozuo 454010, Henan, China; School of Mathematical Sciences, Tianjin Polytechnic University, Tianjin 300387, China}
\email{\href{mailto: F. Qi <qifeng618@gmail.com>}{qifeng618@gmail.com}, \href{mailto: F. Qi <qifeng618@hotmail.com>}{qifeng618@hotmail.com}, \href{mailto: F. Qi <qifeng618@qq.com>}{qifeng618@qq.com}}
\urladdr{\url{https://qifeng618.wordpress.com}}

\author[W.-H. Li]{Wen-Hui Li}
\address[Li]{Department of Fundamental Courses, Zhenghzou University of Science and Technology, Zhengzhou 450064, Henan, China}
\email{\href{mailto: W.-H. Li <wen.hui.li@foxmail.com>}{wen.hui.li@foxmail.com}, \href{mailto: W.-H. Li <wen.hui.li102@gmail.com>}{wen.hui.li102@gmail.com}}
\urladdr{\url{https://orcid.org/0000-0002-1848-8855}}

\author[G.-S. Wu]{Guo-Sheng Wu}
\address[Wu]{School of Computer Science, Sichuan Technology and Business University, Chengdu 611745, Sichuan, China}
\email{mrwuguosheng@sina.com}

\author[B.-N. Guo]{Bai-Ni Guo}
\address[Guo]{School of Mathematics and Informatics, Henan Polytechnic University, Jiaozuo 454010, Henan, China}
\email{\href{mailto: B.-N. Guo <bai.ni.guo@gmail.com>}{bai.ni.guo@gmail.com}, \href{mailto: B.-N. Guo <bai.ni.guo@hotmail.com>}{bai.ni.guo@hotmail.com}}
\urladdr{\url{https://orcid.org/0000-0001-6156-2590}}

\dedicatory{Dedicated to people facing and fighting 2019-nCoV}

\begin{abstract}
In the paper, the authors review several refinements of Young's integral inequality via several mean value theorems, such as Lagrange's and Taylor's mean value theorems of Lagrange's and Cauchy's type remainders, and via several fundamental inequalities, such as \v{C}eby\v{s}ev's integral inequality, Hermite--Hadamard's type integral inequalities, H\"older's integral inequality, and Jensen's discrete and integral inequalities, in terms of higher order derivatives and their norms, survey several applications of several refinements of Young's integral inequality, and further refine Young's integral inequality via P\'olya's type integral inequalities.
\end{abstract}



\maketitle
\tableofcontents

\section{Young's integral inequality and several refinements}

In the first part of this paper, we mainly review several refinements of Young's integral inequality via several mean value theorems, such as Lagrange's and Taylor's mean value theorems of Lagrange's and Cauchy's type remainders, and via several fundamental inequalities, such as \v{C}eby\v{s}ev's integral inequality, Hermite--Hadamard's type integral inequalities, H\"older's integral inequality, and Jensen's discrete and integral inequalities, in terms of higher order derivatives and their norms, and simply survey several applications of several refinements of Young's integral inequality.

\subsection{Young's integral inequality}

One of fundamental and general inequalities in mathematics is Young's integral inequality below.

\begin{thm}[\cite{Young-Roy-London-1912}]\label{Young-orig-thm}
Let $h(x)$ be a real-valued, continuous, and strictly increasing function on $[0,c]$ with $c>0$. If $h(0)=0$, $a\in[0,c]$, and $b\in[0,h(c)]$, then
\begin{equation}\label{Young-eq1}
\int_0^ah(x)\td x+\int_0^bh^{-1}(x)\td x\ge ab,
\end{equation}
where $h^{-1}$ denotes the inverse function of $h$. The equality in~\eqref{Young-eq1} is valid if and only if $b=h(a)$.
\end{thm}

\begin{proof}
This proof is adapted from the proof of~\cite[Section~2.7, Theorem~1]{mit}.
\par
Set
\begin{equation}\label{mit2}
f(a)=ab-\int_0^ah(x)\td x
\end{equation}
and consider $b>0$ as a parameter. Since $f'(a)=b-h(a)$ and $h$ is strictly increasing, one obtains
\begin{equation*}
f'(a)
\begin{cases}
>0, & 0<a<h^{-1}(b);\\
=0, & a=h^{-1}(b);\\
<0, & a>h^{-1}(b).
\end{cases}
\end{equation*}
This means that $f(a)$ has a maximum of $f$ at $a=h^{-1}(b)$. Therefore, it follows that
\begin{equation}\label{mit3}
f(a)\le\max\{f(x)\}=f\bigl(h^{-1}(b)\bigr).
\end{equation}
Integrating by parts gives
\begin{equation*}
f\bigl(h^{-1}(b)\bigr)=bh^{-1}(b)-\int_0^{h^{-1}(b)}h(x)\td x
=\int_{0}^{h^{-1}(b)}xh'(x)\td x.
\end{equation*}
Substituting $y=h(x)$ into the above integral yields
\begin{equation}\label{mit4}
f\bigl(h^{-1}(b)\bigr)=\int_{0}^{b}h^{-1}(y)\td y.
\end{equation}
Putting~\eqref{mit2} and~\eqref{mit4} into~\eqref{mit3} results in~\eqref{Young-eq1}. The proof of Theorem~\ref{Young-orig-thm} is complete.
\end{proof}

\begin{rem}
The geometric interpretation of Young's integral inequality~\eqref{Young-eq1} can be demonstrated by Figures~\ref{Young-Geom-M-1.jpg} and~\ref{Young-Geom-M-2.jpg}.
\begin{figure}[hbtp]
  \centering
  \includegraphics[width=0.8\textwidth]{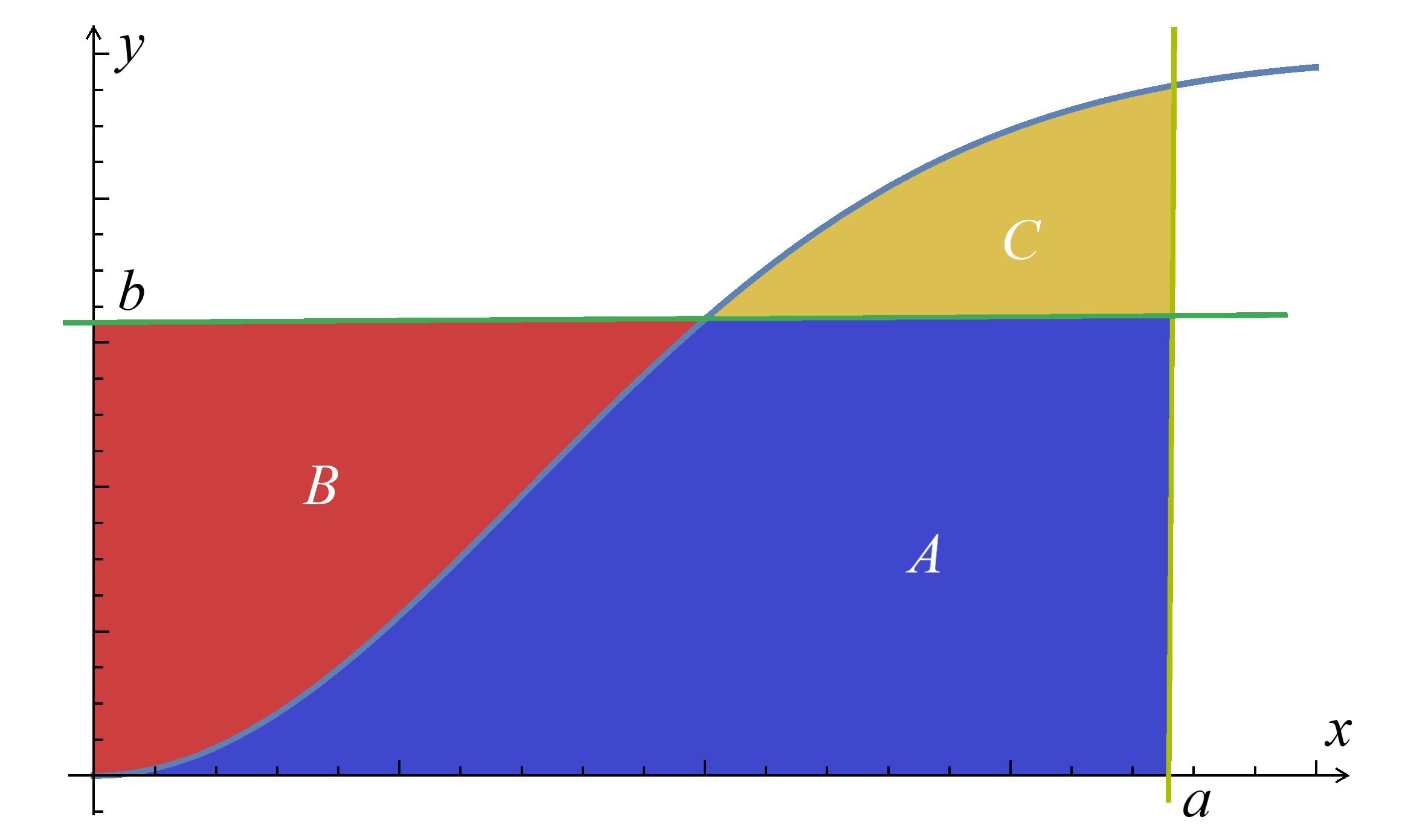}
  \caption{Geometric interpretation of the inequality~\eqref{Young-eq1}}\label{Young-Geom-M-1.jpg}
\end{figure}
\begin{figure}[hbtp]
  \centering
  \includegraphics[width=0.8\textwidth]{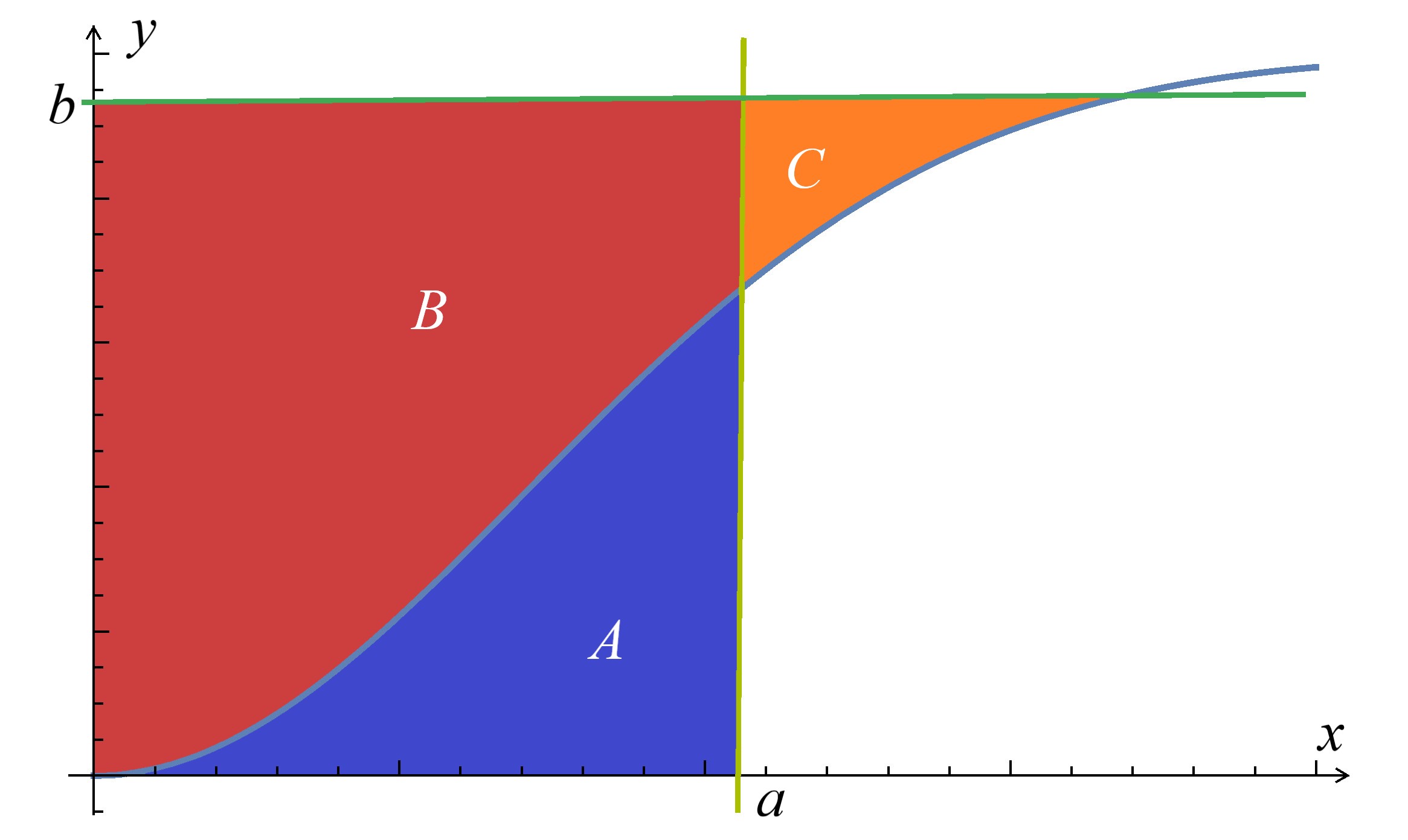}
  \caption{Geometric interpretation of the inequality~\eqref{Young-eq1}}\label{Young-Geom-M-2.jpg}
\end{figure}
\par
In Figure~\ref{Young-Geom-M-1.jpg}, we have
\begin{gather*}
A+C=\int_0^ah(x)\td x,\quad A+B=ab,\quad B=\int_0^bh^{-1}(x)\td x,\\
A+B+C=\int_0^ah(x)\td x+\int_0^bh^{-1}(x)\td x\ge ab=A+B.
\end{gather*}
Therefore, the inequality~\eqref{Young-eq1} means that the area
\begin{equation}\label{C-area-exp-1}
C=\int_{h^{-1}(b)}^{a}h(x)\td x-b\bigl[a-h^{-1}(b)\bigr]\ge0.
\end{equation}
\par
In Figure~\ref{Young-Geom-M-2.jpg}, we have
\begin{gather*}
A=\int_0^ah(x)\td x,\quad A+B=ab,\quad B+C=\int_0^bh^{-1}(x)\td x,\\
A+B+C=\int_0^ah(x)\td x+\int_0^bh^{-1}(x)\td x\ge ab=A+B.
\end{gather*}
Therefore, the inequality~\eqref{Young-eq1} means that the area
\begin{equation}\label{C-area-exp-2}
C=b\bigl[h^{-1}(b)-a\bigr]-\int^{h^{-1}(b)}_{a}h(x)\td x\ge0.
\end{equation}
\end{rem}

\begin{rem}\label{C-area-exp-unif-rem}
We notice that two expressions~\eqref{C-area-exp-1} and~\eqref{C-area-exp-2} are of the same form
\begin{equation}\label{C-area-exp-unif}
C=\int_{h^{-1}(b)}^{a}h(x)\td x-ab+bh^{-1}(b)\ge0,
\end{equation}
no matter which of $a$ and $h^{-1}(b)$ is smaller or bigger.
\end{rem}

\begin{rem}
When $p>1$, taking $h(x)=x^{p-1}$ in~\eqref{Young-eq1} derives
\begin{equation*}
\frac{1}{p}a^p+\frac{1}{q}b^{q}\ge ab
\end{equation*}
for $a,b\ge0$ and $p,q>1$ satisfying $\frac1p+\frac1q=1$. Further replacing $a^p$ and $b^q$ by $x$ and $y$ respectively leads to
\begin{equation}\label{weight-arith-geom-ineq}
x^{1/p}y^{1/q}\le\frac{x}{p}+\frac{y}{q}
\end{equation}
for $x,y\ge0$ and $a,q>1$ satisfying $\frac{1}{p}+\frac{1}{q}=1$. Perhaps this is why the weighted arithmetic-geometric inequality~\eqref{weight-arith-geom-ineq} is also called Young's inequality in~\cite{Dragomir-RACSAM-2017, Dragomir-2018-spec-matrices, Korus-Hungar-2017} and closely related references therein.
\end{rem}

\begin{rem}
The inequality
\begin{equation*}
\sum_{k=1}^n\frac{\cos(k\theta)}{k}>-1,\quad n\ge2,\quad \theta\in[0,\pi]
\end{equation*}
is also called Young's inequality in~\cite{alzer-kwong-jmaa-2019, alzer-koumandos-edinb-2007} and closely related references therein.
\end{rem}

\begin{rem}
In~\cite[Secton~2.7]{mit} and~\cite[Chapter~XIV]{mpf-1993}, a plenty of refinements, extensions, generalizations, and applications of Young's integral inequality~\eqref{Young-eq1} were collected, reviewed, and surveyed. For some new and recent development on this topic after 1990, please refer to the papers~\cite{Anderson-jipam, Ruthing-Young, T.32, W.06, Zhu-Young} and closely related references therein.
\end{rem}

\subsection{Refinements of Young's integral inequality via Lagrange's mean value theorem}

In 2008, Hoorfar and Qi refined Young's integral inequality~\eqref{Young-eq1} via Lagrange's mean value theorem for derivatives.

\begin{thm}[{\cite[Theorem~1]{Young-Hoorfar-Qi.tex}}]\label{hoorfar-qi-young}
Let $h(x)$ be a differentiable and strictly increasing function on $[0,c]$ for $c>0$ and let $h^{-1}$ be the inverse function of $h$. If $h(0)=0$, $a\in[0,c]$, $b\in[0,h(c)]$, and $h'(x)$ is strictly monotonic on $[0,c]$, then
\begin{equation}\label{Hoofar-Qi-eq4}
\frac{m}{2}\bigl[a-h^{-1}(b)\bigr]^{2}
\le\int_0^ah(x)\td x+\int_0^bh^{-1}(x)\td x-ab
\le\frac{M}{2}\bigl[a-h^{-1}(b)\bigr]^{2},
\end{equation}
where
\begin{equation*}
m=\min\bigl\{h'(a), h'\bigl(h^{-1}(b)\bigr)\bigr\}
\end{equation*}
and
\begin{equation*}
M=\max\bigl\{h'(a), h'\bigl(h^{-1}(b)\bigr)\bigr\}.
\end{equation*}
The equalities in~\eqref{Hoofar-Qi-eq4} are valid if and only if $b=h(a)$.
\end{thm}

\begin{proof}
This is a modification of the proof of~\cite[Theorem~1]{Young-Hoorfar-Qi.tex} in~\cite[Section~2]{Young-Hoorfar-Qi.tex}.
\par
Changing the variable of integration by $x=h(y)$ and integrating by parts yield
\begin{equation}\label{hoorfar-qi-eq5}
\begin{aligned}
\int^{a}_{0}h(x)\td x+\int^{b}_{0}h^{-1}(x)\td x
&=\int^{a}_{0}h(x)\td x+\int^{h^{-1}(b)}_{0}yh'(y)\td y \\
&=\int^{a}_{0}h(x)\td x +bh^{-1}(b)-\int^{h^{-1}(b)}_{0}h(x)\td x\\
&=bh^{-1}(b)+\int^{a}_{h^{-1}(b)}h(x)\td x\\
&=ab+\int^{a}_{h^{-1}(b)}[h(x)-b]\td x.
\end{aligned}
\end{equation}
From the last line in~\eqref{hoorfar-qi-eq5}, we can see that, if $h^{-1}(b)=a$, then those equalities in~\eqref{Hoofar-Qi-eq4} hold.
\par
If $h^{-1}(b)<a$, since $h(x)$ is strictly increasing, then $h(x)-b>0$ for $x\in\bigl(h^{-1}(b),a\bigr)$. By Lagrange's mean value theorem for derivatives, we can see that there exists $\xi=\xi(x)$, satisfying $h^{-1}(b)<\xi<x\le a$, such that
$$
0<h(x)-b =\bigl[x-h^{-1}(b)\bigr]h'(\xi).
$$
By virtue of monotonicity of $h'(x)$ on $[0,c]$, we reveal that
\begin{equation*}
0<m=\min\bigl\{h'(a),h'\bigl(h^{-1}(b)\bigr)\bigr\} <h'(\xi)<\max\bigl\{h'(a), h'\bigl(h^{-1}(b)\bigr)\bigr\}=M.
\end{equation*}
Consequently, we have
\begin{equation*}
0<m\bigl[x-h^{-1}(b)\bigr]<h(x)-b<M\bigl[x-h^{-1}(b)\bigr].
\end{equation*}
As a result, we have
\begin{equation*}
m\int^{a}_{h^{-1}(b)}\bigl[x-h^{-1}(b)\bigr]\td x
<\int^{a}_{h^{-1}(b)} [h(x)-b]\td x
<M\int^{a}_{h^{-1}(b)}\bigl[x-h^{-1}(b)\bigr]\td x
\end{equation*}
which is equivalent to
\begin{equation}\label{eq6}
\frac{m}{2}\bigl[a-h^{-1}(b)\bigr]^{2} <\int^{a}_{h^{-1}(b)}[h(x)-b]\td x<\frac{M}{2}\bigl[a-h^{-1}(b)\bigr]^{2}.
\end{equation}
\par
If $h^{-1}(b)>a$, we can derive inequalities in~\eqref{eq6} by a similar argument as above.
\par
Substituting the double inequality~\eqref{eq6} into the equality~\eqref{hoorfar-qi-eq5} leads to the double inequality~\eqref{Hoofar-Qi-eq4}. The proof of Theorem~\ref{hoorfar-qi-young} is complete.
\end{proof}

\begin{rem}
The geometric interpretation of the double inequality~\eqref{Hoofar-Qi-eq4} is that the areas $C$ in Figures~\ref{Young-Geom-C-1-2.jpg} to~\ref{Young-Geom-C-2-1.jpg} satisfy
\begin{equation}\label{Area-C-Ineq}
\frac{m}{2}\bigl[a-h^{-1}(b)\bigr]^{2}
\le C
\le\frac{M}{2}\bigl[a-h^{-1}(b)\bigr]^{2}.
\end{equation}
When $h'(x)$ is strictly increasing, the double inequality~\eqref{Area-C-Ineq} can be equivalently written as
\begin{equation*}
\frac{h'(a)}{2}\bigl[a-h^{-1}(b)\bigr]
\le \frac{\int_{h^{-1}(b)}^{a}h(x)\td x}{a-h^{-1}(b)}
\le\frac{h'\bigl(h^{-1}(b)\bigr)}{2}\bigl[a-h^{-1}(b)\bigr]
\end{equation*}
and
\begin{equation*}
\frac{h'(a)}{2}\bigl[h^{-1}(b)-a\bigr]
\le \frac{\int^{h^{-1}(b)}_{a}h(x)\td x}{h^{-1}(b)-a}
\le\frac{h'\bigl(h^{-1}(b)\bigr)}{2}\bigl[h^{-1}(b)-a\bigr]
\end{equation*}
corresponding to Figures~\ref{Young-Geom-C-1-2.jpg} and~\ref{Young-Geom-C-2-2.jpg} respectively.
\begin{figure}[htbp]
  \centering
  \includegraphics[width=0.8\textwidth]{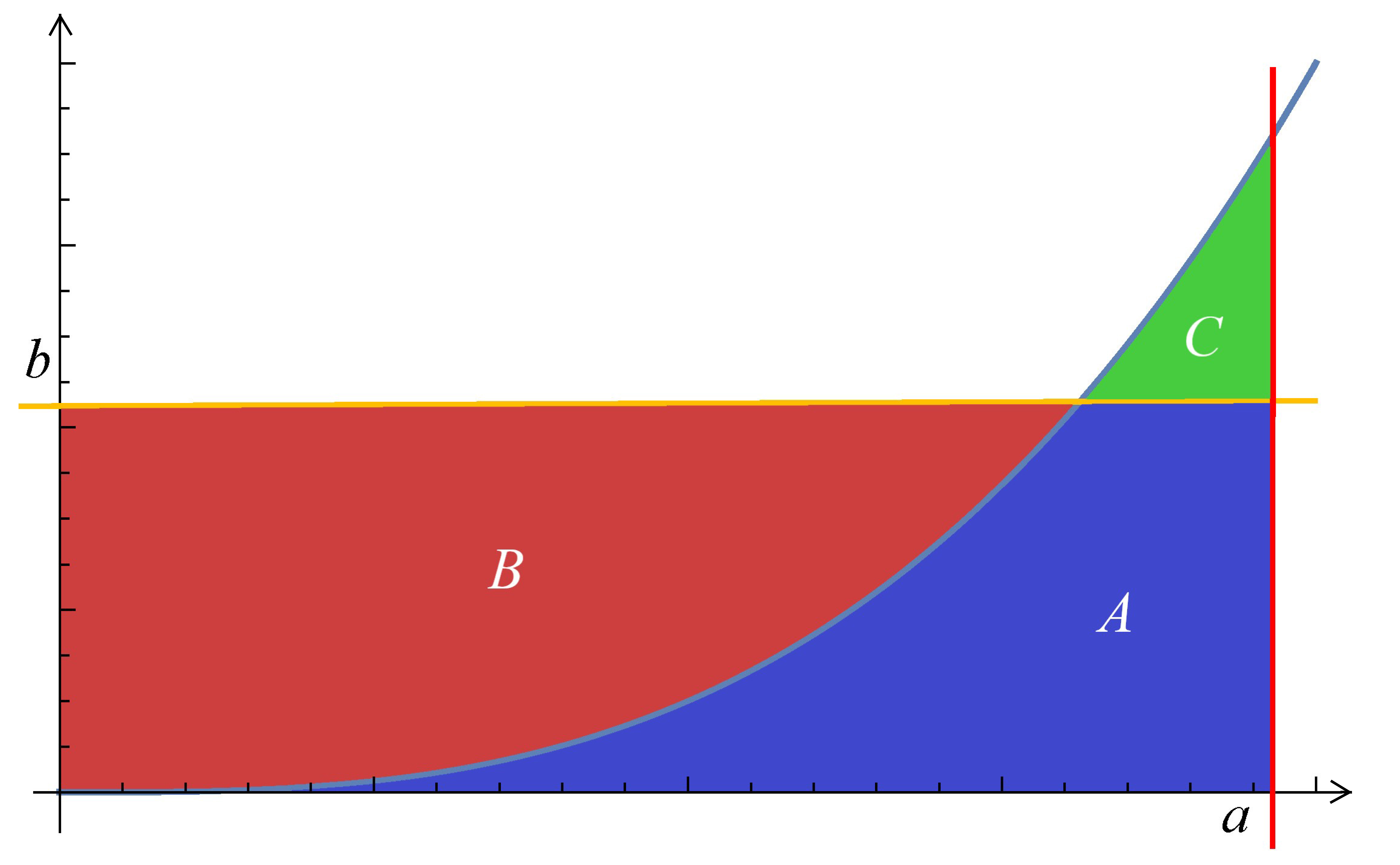}
  \caption{Geometric interpretation of the double inequality~\eqref{Hoofar-Qi-eq4}}\label{Young-Geom-C-1-2.jpg}
\end{figure}
\begin{figure}[hbtp]
  \centering
  \includegraphics[width=0.8\textwidth]{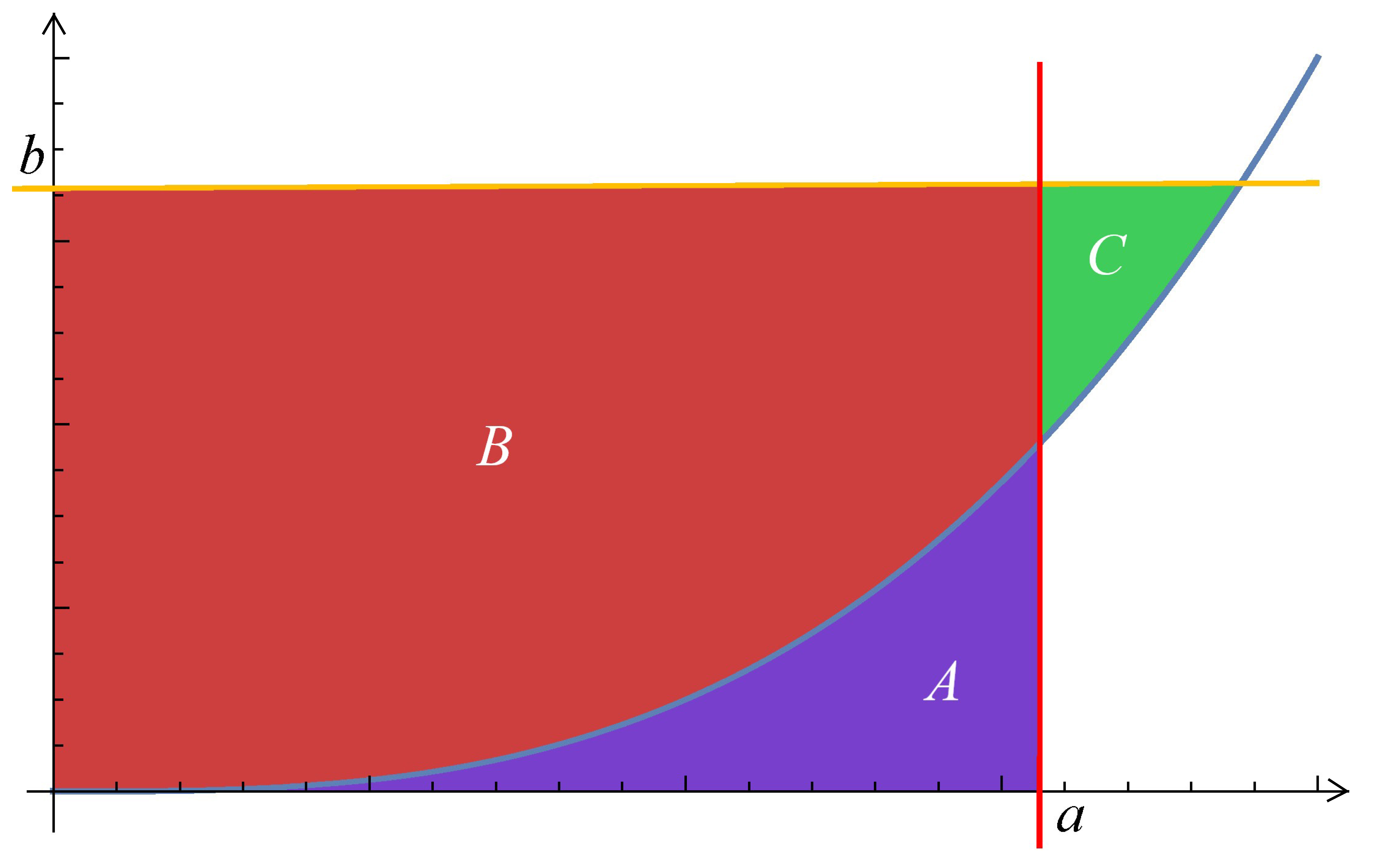}
  \caption{Geometric interpretation of the double inequality~\eqref{Hoofar-Qi-eq4}}\label{Young-Geom-C-2-2.jpg}
\end{figure}
\begin{figure}[hbtp]
  \centering
  \includegraphics[width=0.8\textwidth]{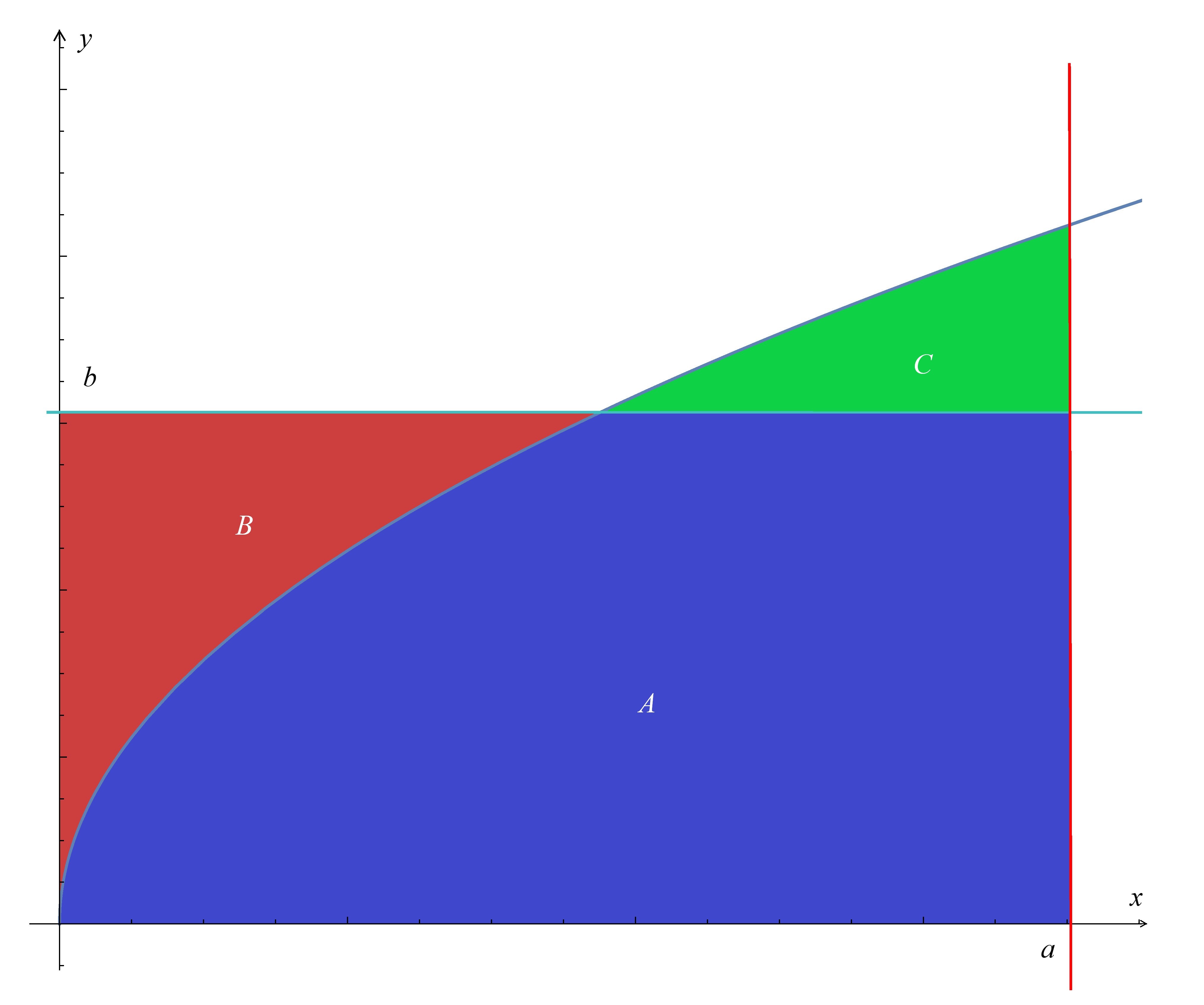}
  \caption{Geometric interpretation of the double inequality~\eqref{Hoofar-Qi-eq4}}\label{Young-Geom-C-1-1.jpg}
\end{figure}
\begin{figure}[hbtp]
  \centering
  \includegraphics[width=0.8\textwidth]{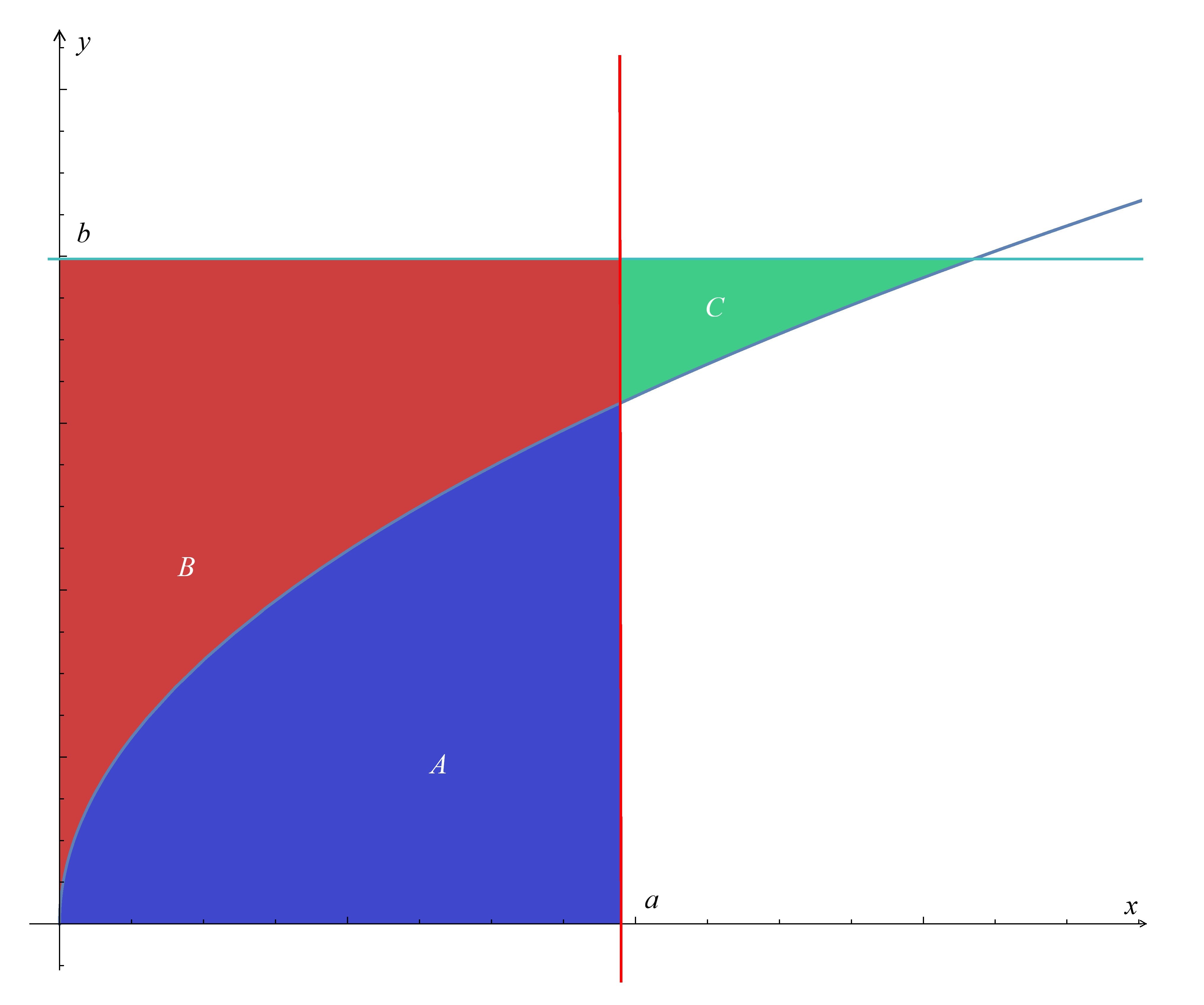}
  \caption{Geometric interpretation of the double inequality~\eqref{Hoofar-Qi-eq4}}\label{Young-Geom-C-2-1.jpg}
\end{figure}
\end{rem}

\begin{rem}
If $Q$ is a $\begin{matrix}\text{convex}\\ \text{concave}\end{matrix}$ function on $J$, then
\begin{equation}\label{HH-Ineq-Eq}
Q\biggl(\frac{\tau+\mu}{2}\biggr)\lesseqgtr\frac{1}{\mu-\tau}\int_{\tau}^{\mu}Q(x)\td x\lesseqgtr\frac{Q(\tau)+Q(\mu)}{2},
\end{equation}
where $J\subseteq\mathbb{R}$ is a nonempty interval and $\tau,\mu\in J$ with $\tau<\mu$. The double inequality~\eqref{HH-Ineq-Eq} is called Hermite--Hadamard's integral inequality for convex functions~\cite{JIAP-D-19-00904.tex, Psh-Kurd-JIA.tex, mathematics-360953.tex}.
When $a>h^{-1}(b)$, as showed in Figures~\ref{Young-Geom-C-1-2.jpg} and~\ref{Young-Geom-C-1-1.jpg}, and $h'(x)$ is strictly increasing, that is, the function $h(x)$ is convex, as showed in Figures~\ref{Young-Geom-C-1-2.jpg} and~\ref{Young-Geom-C-2-2.jpg}, applying the double inequality~\eqref{HH-Ineq-Eq} yields
\begin{equation*}
h\biggl(\frac{a+h^{-1}(b)}2\biggr)
\le \frac{\int_{h^{-1}(b)}^{a}h(x)\td x}{a-h^{-1}(b)}
\le\frac{h(a)+b}{2}.
\end{equation*}
Substituting this into the third line in~\eqref{hoorfar-qi-eq5} gives
\begin{gather*}
\int^{a}_{0}h(x)\td x+\int^{b}_{0}h^{-1}(x)\td x
\le ab+\frac{h(a)-b}{2}\bigl[a-h^{-1}(b)\bigr]
\end{gather*}
and
\begin{gather*}
\int^{a}_{0}h(x)\td x+\int^{b}_{0}h^{-1}(x)\td x
\ge bh^{-1}(b)+h\biggl(\frac{a+h^{-1}(b)}2\biggr)\bigl[a-h^{-1}(b)\bigr]\\
=ab+\biggl[h\biggl(\frac{a+h^{-1}(b)}2\biggr)-b\biggr]\bigl[a-h^{-1}(b)\bigr].
\end{gather*}
Equivalently speaking, it follows that the area $C$ satisfies
\begin{equation*}
\biggl[h\biggl(\frac{a+h^{-1}(b)}2\biggr)-b\biggr]\bigl[a-h^{-1}(b)\bigr]\le C
\le \frac{h(a)-b}{2}\bigl[a-h^{-1}(b)\bigr].
\end{equation*}
Similarly, we can discuss other cases, corresponding to Figures~\ref{Young-Geom-C-1-1.jpg} and~\ref{Young-Geom-C-2-1.jpg}, that the derivative $h'(x)$ is strictly decreasing.
\end{rem}

\begin{rem}
Mercer has applied and employed the double inequality~\eqref{Hoofar-Qi-eq4} in the paper~\cite{Mercer-2008-JMI} and in the Undergraduate Texts in Mathematics~\cite{Mercer-2014-UTM}.
\end{rem}

\subsection{Refinements of Young's integral inequality via Hermite--Hadamard's and \v{C}eby\v{s}ev's integral inequalities}

In 2009 and 2010, among other things, Jak\v{s}eti\'c and Pe\v{c}ari\'c refined Young's integral inequality~\eqref{Young-eq1} and Hoorfar--Qi's double inequality~\eqref{Hoofar-Qi-eq4} in~\cite{Jak-Pecaric-2009-AEJM, MIA-2010-13-43}.

\begin{thm}[{\cite[Theorem~2.1]{Jak-Pecaric-2009-AEJM} and~\cite[Theorem~2.3]{MIA-2010-13-43}}]\label{Jak-Pecaric-2009-Thm}
Let $h(x)$ be a differentiable and strictly increasing function on $[0,c]$ for $c>0$, $h(0)=0$, $a\in[0,c]$, $b\in[0,h(c)]$, and $h^{-1}$ be the inverse function of $h$. Denote
\begin{equation}\label{alpha-beta-def-eq}
\alpha=\min\bigl\{a,h^{-1}(b)\bigr\}\quad\text{and}\quad \beta=\max\bigl\{a,h^{-1}(b)\bigr\}.
\end{equation}
\begin{enumerate}
\item
If $h'(x)$ is increasing on $[\alpha,\beta]$ and $b<h(a)$, or if $h'(x)$ is decreasing on $[\alpha,\beta]$ and $b>h(a)$, then
\begin{equation}
\begin{aligned}\label{theorem2.3-mia}
\bigl[a-h^{-1}(b)\bigr]\biggl[h\biggl(\frac{a+h^{-1}(b)}2\biggr)-b\biggr]
&\le\int_0^ah(x)\td x+\int_0^bh^{-1}(x)\td x-ab\\
&\le\frac12\bigl[a-h^{-1}(b)\bigr][h(a)-b].
\end{aligned}
\end{equation}
\item
If $h'(x)$ is increasing on $[\alpha,\beta]$ and $b>h(a)$, or if $h'(x)$ is decreasing on $[\alpha,\beta]$ and $b<h(a)$, then the inequality~\eqref{theorem2.3-mia} is reversed.
\item
The equality in~\eqref{theorem2.3-mia} is valid if and only if $h(x)=\lambda x$ for $\lambda>0$ or $b=h(a)$.
\end{enumerate}
\end{thm}

\begin{proof}
This is the outline of proofs of~\cite[Theorem~2.1]{Jak-Pecaric-2009-AEJM} and~\cite[Theorem~2.3]{MIA-2010-13-43}.
\par
From the third line in~\eqref{hoorfar-qi-eq5}, it follows that
\begin{equation}\label{Pecaric-Javic-eq}
\int^{a}_{0}h(x)\td x+\int^{b}_{0}h^{-1}(x)\td x-ab
=b\bigl[h^{-1}(b)-a\bigr]+\int^{a}_{h^{-1}(b)}h(x)\td x.
\end{equation}
Considering monotonicity of $h'(x)$ and applying the double inequality~\eqref{HH-Ineq-Eq} to the integrand in the last term of~\eqref{Pecaric-Javic-eq}, we can derive the double inequality~\eqref{theorem2.3-mia}.
\par
The last term in~\eqref{hoorfar-qi-eq5} can be rewritten as
\begin{multline}\label{2.2last-int-rew}
\int^{a}_{h^{-1}(b)}[h(x)-b]\td x
=\int^{a}_{h^{-1}(b)}\bigl[h(x)-h\bigl(h^{-1}(b)\bigr)\bigr]\td x\\
=\int^{a}_{h^{-1}(b)}\int_{h^{-1}(b)}^{x}h'(u)\td u\td x
=\int^{a}_{h^{-1}(b)}(a-u)h'(u)\td u.
\end{multline}
Let $f,g:[\mu,\nu]\to\mathbb{R}$ be integrable functions satisfying that they are both increasing or both decreasing. Then
\begin{equation}\label{Cebysev-Ineq}
\int_\mu^\nu f(x)\td x\int_\mu^\nu g(x)\td x
\le(\nu-\mu)\int_\mu^\nu f(x)g(x)\td x.
\end{equation}
If one of the function $f$ or $g$ is nonincreasing and the other nondecreasing, then the inequality in~\eqref{Cebysev-Ineq} is reversed.
The inequality~\eqref{Cebysev-Ineq} is called \v{C}eby\v{s}ev's integral inequality in the literature~\cite[Chapter~IX]{mpf-1993} and~\cite{tcheby-l.tex, Symmetry-Qi-Du-Nisar.tex}. Applying~\eqref{Cebysev-Ineq} to the last term in~\eqref{2.2last-int-rew} leads to the right hand side of the inequality~\eqref{theorem2.3-mia}.
The proof of Theorem~\ref{Jak-Pecaric-2009-Thm} is complete.
\end{proof}

\begin{rem}
The double inequality~\eqref{theorem2.3-mia} can be geometrically interpreted as
\begin{equation*}
\bigl[a-h^{-1}(b)\bigr]\biggl[h\biggl(\frac{a+h^{-1}(b)}2\biggr)-b\biggr]
\le C
\le\frac12\bigl[a-h^{-1}(b)\bigr][h(a)-b],
\end{equation*}
where $C$ denotes the area showed in Figures~\ref{Young-Geom-M-1.jpg} to~\ref{Young-Geom-C-2-1.jpg}.
\end{rem}

\subsection{Refinements of Young's integral inequality via Jensen's discrete and integral inequalities}
In~\cite[Theorem~2.6]{MIA-2010-13-43}, Jensen's discrete and integral inequalities were employed to establish the following inequalities which refine Young's integral inequality~\eqref{Young-eq1} and Hoorfar--Qi's double inequality~\eqref{Hoofar-Qi-eq4}.

\begin{thm}[{\cite[Theorem~2.6]{MIA-2010-13-43}}]\label{Jak-Pecaric-2010-Thm}
Let $h(x)$ be a differentiable and strictly increasing function on $[0,c]$ for $c>0$ and let $h^{-1}$ be the inverse function of $h$. If $h(0)=0$, $a\in[0,c]$, $b\in[0,h(c)]$, and $h'(x)$ is convex on $[\alpha,\beta]$, then
\begin{equation}
\begin{aligned}\label{theorem2.6-mia}
\frac{\bigl[a-h^{-1}(b)\bigr]^2}{2}h'\biggl(\frac{a+2h^{-1}(b)}{3}\biggr)
&\le \int_0^ah(x)\td x+\int_0^bh^{-1}(x)\td x-ab\\
&\le\frac{\bigl[a-h^{-1}(b)\bigr]^2}3\biggl[\frac{h'(a)}{2}+h'\bigl(h^{-1}(b)\bigr)\biggr].
\end{aligned}
\end{equation}
If $h'(x)$ is concave, then the double inequality~\eqref{theorem2.6-mia} is reversed.
\end{thm}

\begin{proof}
This is the outline of the proof of~\cite[Theorem~2.6]{MIA-2010-13-43}.
\par
Changing the variable of the last term in~\eqref{2.2last-int-rew} results in
\begin{equation}\label{variable-change-eq}
\int^{a}_{h^{-1}(b)}(a-u)h'(u)\td u
=\int_{0}^{1}\bigl[a-h^{-1}(b)\bigr]^2(1-x)h'\bigl(xa+(1-x)h^{-1}(b)\bigr)\td x.
\end{equation}
If $f$ is a convex function on an interval $I\subseteq\mathbb{R}$ and if $n\ge2$ and $x_k\in I$ for $1\le k\le n$, then
\begin{equation}\label{Jensen-Ineq-Eq}
f\Biggl(\frac{1}{\sum_{k=1}^{n}p_k}\sum_{k=1}^{n}p_kx_k\Biggr)\le\frac{1}{\sum_{k=1}^{n}p_k}\sum_{k=1}^{n}p_kf(x_k),
\end{equation}
where $p_k>0$ for $1\le k\le n$. If $f$ is concave, the inequality~\eqref{Jensen-Ineq-Eq} is reversed.
The inequality~\eqref{Jensen-Ineq-Eq} is called Jensen's discrete inequality for convex functions in the literature~\cite[Section~1.4]{mit} and~\cite[Chapter~I]{mpf-1993}.
Applying~\eqref{Jensen-Ineq-Eq} to the third factor in the integrand of the right hand side in~\eqref{variable-change-eq} arrives at the right inequality in~\eqref{theorem2.6-mia}.
\par
Let $\phi$ be a convex function on $[\mu,\nu]$, $f\in L_1(\mu,\nu)$, and $\sigma$ be a non-negative measure. Then
\begin{equation}\label{Jensen-Ineq-Int-Eq}
\phi\Biggl(\frac{\int_{\mu}^{\nu}f(x)\td\sigma}{\int_{\mu}^{\nu}\td\sigma}\Biggr)
\le\frac{\int_{\mu}^{\nu}\phi(f(x))\td\sigma}{\int_{\mu}^{\nu}\td\sigma}.
\end{equation}
If $\phi$ is a concave function, then the inequality~\eqref{Jensen-Ineq-Int-Eq} is reversed. The inequality~\eqref{Jensen-Ineq-Int-Eq} is called Jensen's integral inequality for convex functions in the literature~\cite[p.~10, (7.15)]{mpf-1993}.
Applying~\eqref{Jensen-Ineq-Int-Eq} yields
\begin{align*}
\int_{h^{-1}(b)}^{a}(a-x)h'(x)\td x
&\ge\frac{[1-h^{-1}(b)]^2}{2}
h'\Biggl(\frac{\int_{h^{-1}(b)}^{a}(a-x)x\td x}{\int_{h^{-1}(b)}^{a}(a-x)\td x}\Biggr)\\
&=\frac{[1-h^{-1}(b)]^2}{2}h'\bigg(\frac{a+2h^{-1}(b)}{3}\biggr).
\end{align*}
The proof of Theorem~\ref{Jak-Pecaric-2010-Thm} is complete.
\end{proof}

\begin{rem}
The double inequality~\eqref{theorem2.6-mia} can be geometrically interpreted as
\begin{equation*}
\frac{\bigl[a-h^{-1}(b)\bigr]^2}{2}h'\biggl(\frac{a+2h^{-1}(b)}{3}\biggr)
\le C
\le\frac{\bigl[a-h^{-1}(b)\bigr]^2}3\biggl[\frac{h'(a)}{2}+h'\bigl(h^{-1}(b)\bigr)\biggr],
\end{equation*}
where $C$ denotes the area showed in Figures~\ref{Young-Geom-M-1.jpg} to~\ref{Young-Geom-C-2-1.jpg}.
\end{rem}

\subsection{Refinements of Young's integral inequality via H\"older's integral inequality}
In~\cite[Theorem~2.1]{MIA-2010-13-43}, H\"older's integral inequality was utilized to present the following inequalities, which refine Young's integral inequality~\eqref{Young-eq1} and Hoorfar--Qi's double inequality~\eqref{Hoofar-Qi-eq4}, in terms of norms.

\begin{thm}[{\cite[Theorem~2.1]{MIA-2010-13-43}}]\label{josip-J-thm2.1}
Let $h(x)$ be a differentiable and strictly increasing function on $[0,c]$ for $c>0$ and let $h^{-1}$ be the inverse function of $h$. If $h(0)=0$, $a\in[0,c]$, $b\in[0,h(c)]$, and $h'(x)$ is almost everywhere continuous with respect to Lebesgue measure on $[\alpha,\beta]$, then the double inequality
\begin{equation}\label{theorem2.1-mia}
C_u\|h'\|_v\le \int_0^ah(x)\td x+\int_0^bh^{-1}(x)\td x-ab \le C_p\|h'\|_q
\end{equation}
is valid for all $u,v$ and $p,q$ satisfying
\begin{enumerate}
\item
$\frac1u+\frac1v=1$ for $u,v\in(-\infty,0)\cup(0,1)$, or $(u,v)=(1,-\infty)$, or $(u,v)=(-\infty,1)$;
\item
$\frac1p+\frac1q=1$ for $1<p,q<\infty$, or $(p,q)=(+\infty,1)$, or $(p,q)=(1,+\infty)$;
\end{enumerate}
where
\begin{equation*}
C_r=
\begin{dcases}
\Biggl[\frac{\bigl|a-h^{-1}(b)\bigr|^{r+1}}{r+1}\Biggr]^{1/r}, & r\ne0,\pm\infty;\\
\bigl|a-h^{-1}(b)\bigr|, & r=+\infty;\\
0, & r=-\infty
\end{dcases}
\end{equation*}
and
\begin{equation*}
\|h'\|_r=
\begin{dcases}
\biggl[\int_\alpha^\beta [h'(t)]^r\td t\biggr]^{1/r}, & r\ne0,\pm\infty;\\
\sup\{h'(t),t\in[\alpha,\beta]\}, & r=+\infty;\\
\inf\{h'(t),t\in[\alpha,\beta]\}, & r=-\infty.
\end{dcases}
\end{equation*}
\end{thm}

\begin{proof}
This is the outline of the proof of~\cite[Theorem~2.1]{MIA-2010-13-43}.
\par
Let $\frac1p+\frac1q=1$ with $p>0$ and $p\ne1$, let $f$ and $g$ be real functions on $[\mu,\nu]$, and let $|f|^p$ and $|g|^q$ be integrable on $[\mu,\nu]$.
\begin{enumerate}
\item
If $p>1$, then
\begin{equation}\label{Holder-Ineq-Eq}
\int_{\mu}^{\nu}|f(x)g(x)|\td x\le \biggl[\int_{\mu}^{\nu}|f(x)|^p\td x\biggr]^{1/p} \biggl[\int_{\mu}^{\nu}|g(x)|^q\td x\biggr]^{1/q}.
\end{equation}
The equality in~\eqref{Holder-Ineq-Eq} holds if and only if $A|f(x)|^p=B|g(x)|^q$ almost everywhere for two constants $A$ and $B$.
\item
If $0<p<1$, then the inequality~\eqref{Holder-Ineq-Eq} is reversed.
\end{enumerate}
The inequality~\eqref{Holder-Ineq-Eq} is called H\"older's integral inequality in the lierature~\cite[Chapter~V]{mpf-1993} and~\cite{Sandor-Szabo-JMAA-1996, Tian-Holder-Jia-2011, Tian-Ha-JMI-2017}.
\par
From~\eqref{2.2last-int-rew}, it follows that,
\begin{enumerate}
\item
by a property of definite integrals, we have
\begin{multline*}
\int^{a}_{h^{-1}(b)}(a-u)h'(u)\td u
=\int^{\alpha}_{\beta}|a-u|h'(u)\td u\\
\le\bigl|h^{-1}(b)-a\bigr|\int^{\alpha}_{\beta}h'(u)\td u
=C_\infty\|h'\|_1;
\end{multline*}
\item
by a property of definite integrals, we have
\begin{equation*}
\int^{a}_{h^{-1}(b)}(a-u)h'(u)\td u
=\int^{\alpha}_{\beta}|a-u|h'(u)\td u
\le C_1\|h'\|_\infty;
\end{equation*}
\item
by H\"older's integral inequality~\eqref{Holder-Ineq-Eq}, we have
\begin{multline*}
\int^{a}_{h^{-1}(b)}(a-u)h'(u)\td u
=\int^{\alpha}_{\beta}|a-u|h'(u)\td u\\
\le\biggl(\int^{\alpha}_{\beta}|a-u|^q\td u\biggr)^{1/q} \biggl(\int^{\alpha}_{\beta}[h'(u)]^p\td u\biggr)^{1/p}
=C_q\|h'\|_p.
\end{multline*}
\end{enumerate}
The rest proofs are straightforward. The proofs of the double inequality~\eqref{theorem2.1-mia} and Theorem~\ref{josip-J-thm2.1} are thus complete.
\end{proof}

\begin{rem}
The double inequality~\eqref{theorem2.1-mia} can be geometrically interpreted as
\begin{equation*}
C_u\|h'\|_v\le C \le C_p\|h'\|_q,
\end{equation*}
where $C$ denotes the area showed in Figures~\ref{Young-Geom-M-1.jpg} to~\ref{Young-Geom-C-2-1.jpg}.
\end{rem}

\subsection{Refinements of Young's integral inequality via Taylor's mean value theorem of Lagrange's type remainder}
In~\cite[Theorem~3.1]{Young-Ineq.tex}, making use of Taylor's mean value theorem of Lagrange's type remainder, Wang, Guo, and Qi refined the above inequalities of Young's type via higher order derivatives.

\begin{thm}[{\cite[Theorem~3.1]{Young-Ineq.tex}}]\label{Qi-Grad-Thm1}
Let $h(0)=0$ and $h(x)$ be strictly increasing on $[0,c]$ for $c>0$, let $h^{(n)}(x)$ for $n\ge0$ be continuous on $[0,c]$, let $h^{(n+1)}(x)$ be finite and strictly monotonic on $(0,c)$, and let $h^{-1}$ be the inverse function of $h$. For $a\in[0,c]$ and $b\in[0,h(c)]$,
\begin{enumerate}
\item
if $b<h(a)$, then
\begin{equation}\label{Qi-Grad-Ineq1}
\begin{aligned}
&\quad\sum_{k=1}^nh^{(k)}\bigl(h^{-1}(b)\bigr)\frac{\bigl[a-h^{-1}(b)\bigr]^{k+1}}{(k+1)!} +m_n(a,b)\frac{\bigl[a-h^{-1}(b)\bigr]^{n+2}}{(n+2)!}\\
&\le\int_0^ah(x)\td x+\int_0^bh^{-1}(x)\td x-ab\\
&\le \sum_{k=1}^nh^{(k)}\bigl(h^{-1}(b)\bigr)\frac{\bigl[a-h^{-1}(b)\bigr]^{k+1}}{(k+1)!} +M_n(a,b)\frac{\bigl[a-h^{-1}(b)\bigr]^{n+2}}{(n+2)!},
\end{aligned}
\end{equation}
where
\begin{equation*}
m_n(a,b)=\min\bigl\{h^{(n+1)}\bigl(h^{-1}(b)\bigr),h^{(n+1)}(a)\bigr\}
\end{equation*}
and
\begin{equation*}
M_n(a,b)=\max\bigl\{h^{(n+1)}\bigl(h^{-1}(b)\bigr),h^{(n+1)}(a)\bigr\};
\end{equation*}
\item
if $b>h(a)$, then
\begin{enumerate}
\item
when $n=2\ell$ for $\ell\ge0$, the double inequality~\eqref{Qi-Grad-Ineq1} is valid;
\item
when $n=2\ell+1$ for $\ell\ge0$, we have
\begin{equation}\label{Qi-Grad-Ineq2}
\begin{aligned}
&\quad\sum_{k=1}^nh^{(k)}\bigl(h^{-1}(b)\bigr)\frac{\bigl[a-h^{-1}(b)\bigr]^{k+1}}{(k+1)!} -M_n(a,b)\frac{\bigl[a-h^{-1}(b)\bigr]^{n+2}}{(n+2)!}\\
&\le\int_0^ah(x)\td x+\int_0^bh^{-1}(x)\td x-ab\\
&\le\sum_{k=1}^nh^{(k)}\bigl(h^{-1}(b)\bigr)\frac{\bigl[a-h^{-1}(b)\bigr]^{k+1}}{(k+1)!} -m_n(a,b)\frac{\bigl[a-h^{-1}(b)\bigr]^{n+2}}{(n+2)!};
\end{aligned}
\end{equation}
\end{enumerate}
\item
if, and only if, $b=h(a)$, those equalities in~\eqref{Qi-Grad-Ineq1} and~\eqref{Qi-Grad-Ineq2} hold.
\end{enumerate}
\end{thm}

\begin{proof}
This is the outline of the proof of~\cite[Theorem~3.1]{Young-Ineq.tex}.
\par
Let $f(x)$ be a function having finite $n$th derivative $f^{(n)}(x)$ everywhere in an open interval $(\mu,\nu)$ and assume that $f^{(n-1)}(x)$ is continuous on the closed interval $[\mu,\nu]$. Then, for a fixed point $x_0\in[\mu,\nu]$ and every $x\in[\mu,\nu]$ with $x\ne x_0$, there exists a point $x_1$ interior to the interval jointing $x$ and $x_0$ such that
\begin{equation}\label{Taylor-thm-Lagrange}
f(x)=f(x_0)+\sum_{k=1}^{n-1}\frac{f^{(k)}(x_0)}{k!}(x-x_0)^k+\frac{f^{(n)}(x_1)}{n!}(x-x_0)^n.
\end{equation}
The formula~\eqref{Taylor-thm-Lagrange} is called Taylor's mean value theorem of Lagrange's type remainder in the literature~\cite[p.~113, Theorem~5.19]{Apostol-MA-1974}. Applying~\eqref{Taylor-thm-Lagrange} in the last term of~\eqref{hoorfar-qi-eq5} reveals
\begin{align*}
\int^{a}_{h^{-1}(b)}[h(x)-b]\td x
&=\int^{a}_{h^{-1}(b)}\bigl[h(x)-h\bigl(h^{-1}(b)\bigr)\bigr]\td x\\
&=\sum_{k=1}^n\frac{h^{(k)}\bigl(h^{-1}(b)\bigr)}{k!}\int_{h^{-1}(b)}^a\bigl[x-h^{-1}(b)\bigr]^k\td x\\ &\quad+\frac1{(n+1)!}\int_{h^{-1}(b)}^ah^{(n+1)}(\xi)\bigl[x-h^{-1}(b)\bigr]^{n+1}\td x\\
&=\sum_{k=1}^nh^{(k)}\bigl(h^{-1}(b)\bigr)\frac{\bigl[a-h^{-1}(b)\bigr]^{k+1}}{(k+1)!}\\
&\quad+\int_{h^{-1}(b)}^ah^{(n+1)}(\xi)\frac{\bigl[x-h^{-1}(b)\bigr]^{n+1}}{(n+1)!}\td x,
\end{align*}
where $\xi$ is a point interior to the interval jointing $x$ and $h^{-1}(b)$. The rest proofs are straightforward discussions on various cases of the factor $h^{(n+1)}(\xi)$. The proof of Theorem~\ref{Qi-Grad-Thm1} is complete.
\end{proof}

\begin{rem}
The double inequalities~\eqref{Qi-Grad-Ineq1} and~\eqref{Qi-Grad-Ineq2} can be geometrically interpreted as
\begin{multline*}
\sum_{k=1}^nh^{(k)}\bigl(h^{-1}(b)\bigr)\frac{\bigl[a-h^{-1}(b)\bigr]^{k+1}}{(k+1)!} +m_n(a,b)\frac{\bigl[a-h^{-1}(b)\bigr]^{n+2}}{(n+2)!}
\le C\\
\le \sum_{k=1}^nh^{(k)}\bigl(h^{-1}(b)\bigr)\frac{\bigl[a-h^{-1}(b)\bigr]^{k+1}}{(k+1)!} +M_n(a,b)\frac{\bigl[a-h^{-1}(b)\bigr]^{n+2}}{(n+2)!}
\end{multline*}
and
\begin{multline*}
\sum_{k=1}^nh^{(k)}\bigl(h^{-1}(b)\bigr)\frac{\bigl[a-h^{-1}(b)\bigr]^{k+1}}{(k+1)!} -M_n(a,b)\frac{\bigl[a-h^{-1}(b)\bigr]^{n+2}}{(n+2)!}
\le C\\
\le\sum_{k=1}^nh^{(k)}\bigl(h^{-1}(b)\bigr)\frac{\bigl[a-h^{-1}(b)\bigr]^{k+1}}{(k+1)!} -m_n(a,b)\frac{\bigl[a-h^{-1}(b)\bigr]^{n+2}}{(n+2)!},
\end{multline*}
where $C$ denotes the area showed in Figures~\ref{Young-Geom-M-1.jpg} to~\ref{Young-Geom-C-2-1.jpg}.
\end{rem}

\subsection{Refinements of Young's integral inequality via Taylor's mean value theorem of Cauchy's type remainder and H\"older's integral inequality}
In~\cite[Theorem~3.2]{Young-Ineq.tex}, employing Taylor's mean value theorem of Cauchy's type remainder and H\"older's integral inequality, Wang, Guo, and Qi refined the above inequalities of Young's type via norms of higher order derivatives.

\begin{thm}[{\cite[Theorem~3.2]{Young-Ineq.tex}}]\label{Qi-Grad-Thm2}
Let $n\ge0$ and $h(x)\in C^{n+1}[0,c]$ such that $h(0)=0$, $h^{(n+1)}(x)\ge0$ on $[\alpha,\beta]$, and $h(x)$ is strictly increasing on $[0,c]$ for $c>0$, let $h^{-1}$ be the inverse function of $h$, and let $a\in[0,c]$ and $b\in[0,h(c)]$. Then
\begin{enumerate}
\item
when $b>h(a)$ and $n=2\ell$ for $\ell\ge0$ or when $b<h(a)$, we have
\begin{align*}
\frac{C_{u,n}}{(n+1)!}\bigl\|h^{(n+1)}\bigr\|_v
&\le\int_0^ah(x)\td x+\int_0^bh^{-1}(x)\td x-ab\\
&\quad-\sum_{k=1}^nh^{(k)}\bigl(h^{-1}(b)\bigr)\frac{\bigl[a-h^{-1}(b)\bigr]^{k+1}}{(k+1)!}\\
&\le\frac{C_{p,n}}{(n+1)!}\bigl\|h^{(n+1)}\bigr\|_q;
\end{align*}
\item
when $b>h(a)$ and $n=2\ell+1$ for $\ell\ge0$, we have
\begin{align*}
-\frac{C_{p,n}}{(n+1)!}\bigl\|h^{(n+1)}\bigr\|_q
&\le\int_0^ah(x)\td x+\int_0^bh^{-1}(x)\td x-ab\\
&\quad-\sum_{k=1}^nh^{(k)}\bigl(h^{-1}(b)\bigr)\frac{\bigl[a-h^{-1}(b)\bigr]^{k+1}}{(k+1)!}\\
&\le-\frac{C_{u,n}}{(n+1)!}\bigl\|h^{(n+1)}\bigr\|_v;
\end{align*}
\end{enumerate}
where $\alpha,\beta$ are defined as in~\eqref{alpha-beta-def-eq},
\begin{align*}
C_{r,n}&=
\begin{dcases}
\Biggl[\frac{\bigl|a-h^{-1}(b)\bigr|^{r(n+1)+1}}{r(n+1)+1}\Biggr]^{1/r}, & r\ne0,\pm\infty;\\
\bigl|a-h^{-1}(b)\bigr|^{n+1}, & r=+\infty;\\
0, & r=-\infty,
\end{dcases}\\
\bigl\|h^{(n+1)}\bigr\|_r&=
\begin{dcases}
\biggl[\int_\alpha^\beta\bigl[h^{(n+1)}(t)\bigr]^r\td t\biggr]^{1/r}, & r\ne0,\pm\infty;\\
\sup\bigl\{h^{(n+1)}(t),t\in[\alpha,\beta]\bigr\}, & r=+\infty;\\
\inf\bigl\{h^{(n+1)}(t),t\in[\alpha,\beta]\bigr\}, & r=-\infty,
\end{dcases}
\end{align*}
and $u,v,p,q$ satisfy
\begin{enumerate}
\item
$u<1$ and $u\ne0$ with $\frac1u+\frac1v=1$, or $(u,v)=(-\infty,1)$, or $(u,v)=(1,-\infty)$;
\item
$1<p,q<\infty$ with $\frac1p+\frac1q=1$, or $(p,q)=(+\infty,1)$, or $(p,q)=(1,+\infty)$.
\end{enumerate}
\end{thm}

\begin{proof}
This is the outline of the proof of~\cite[Theorem~3.2]{Young-Ineq.tex}.
\par
If $f(x)\in C^{n+1}[\mu,\nu]$ and $x_0\in[\mu,\nu]$, then
\begin{equation}\label{Taylor-thm-Cauchy}
f(x)=\sum_{k=0}^n\frac{f^{(k)}(x_0)}{k!}(x-x_0)^k+\frac{1}{n!}\int_{x_0}^{x}(x-t)^nf^{(n+1)}(t)\td t.
\end{equation}
The formula~\eqref{Taylor-thm-Cauchy} is called Taylor's mean value theorem of Cauchy's type remainder in the literature~\cite[p.~279, Theorem~7.6]{Apostol-Cal-1967} and~\cite[p.~6, 1.4.37]{NIST-HB-2010}. Applying the formula~\eqref{Taylor-thm-Cauchy} to the integrand in the last term of~\eqref{hoorfar-qi-eq5} yields
\begin{align*}
\int^{a}_{h^{-1}(b)}[h(x)-b]\td x
&=\int^{a}_{h^{-1}(b)}\bigl[h(x)-h\bigl(h^{-1}(b)\bigr)\bigr]\td x\\
&=\sum_{k=1}^nh^{(k)}\bigl(h^{-1}(b)\bigr)\frac{\bigl[a-h^{-1}(b)\bigr]^{k+1}}{(k+1)!}\\
&\quad+\int_{h^{-1}(b)}^a\frac{1}{n!}\int_{h^{-1}(b)}^{x}(x-t)^nh^{(n+1)}(t)\td t\td x\\
&=\sum_{k=1}^nh^{(k)}\bigl(h^{-1}(b)\bigr)\frac{\bigl[a-h^{-1}(b)\bigr]^{k+1}}{(k+1)!}\\
&\quad+\int_{h^{-1}(b)}^a\frac{1}{n!}\int_t^{a}(x-t)^nh^{(n+1)}(t)\td x\td t\\
&=\sum_{k=1}^nh^{(k)}\bigl(h^{-1}(b)\bigr)\frac{\bigl[a-h^{-1}(b)\bigr]^{k+1}}{(k+1)!}\\
&\quad+\frac{1}{(n+1)!}\int_{h^{-1}(b)}^a(a-t)^{n+1}h^{(n+1)}(t)\td t\\
&=\sum_{k=1}^nh^{(k)}\bigl(h^{-1}(b)\bigr)\frac{\bigl[a-h^{-1}(b)\bigr]^{k+1}}{(k+1)!}\\
&\quad+
\begin{dcases}
\frac{1}{{(n+1)!}}\int_\alpha^\beta|a-t|^{n+1}h^{(n+1)}(t)\td t, & b<h(a);\\
\frac{(-1)^n}{{(n+1)!}}\int_\alpha^\beta|a-t|^{n+1}h^{(n+1)}(t)\td t, & b>h(a).
\end{dcases}
\end{align*}
Discussing and making use of H\"older's integral inequality~\eqref{Holder-Ineq-Eq} as in the proof of Theorem~\ref{josip-J-thm2.1}, we can complete the proof of Theorem~\ref{Qi-Grad-Thm2}.
\end{proof}

\begin{rem}
Two double inequalities in Theorem~\ref{Qi-Grad-Thm2} can be geometrically interpreted as
\begin{align*}
\frac{C_{u,n}}{(n+1)!}\bigl\|h^{(n+1)}\bigr\|_v
&\le C-\sum_{k=1}^nh^{(k)}\bigl(h^{-1}(b)\bigr)\frac{\bigl[a-h^{-1}(b)\bigr]^{k+1}}{(k+1)!}\\
&\le\frac{C_{p,n}}{(n+1)!}\bigl\|h^{(n+1)}\bigr\|_q
\end{align*}
and
\begin{align*}
-\frac{C_{p,n}}{(n+1)!}\bigl\|h^{(n+1)}\bigr\|_q
&\le C-\sum_{k=1}^nh^{(k)}\bigl(h^{-1}(b)\bigr)\frac{\bigl[a-h^{-1}(b)\bigr]^{k+1}}{(k+1)!}\\
&\le-\frac{C_{u,n}}{(n+1)!}\bigl\|h^{(n+1)}\bigr\|_v,
\end{align*}
where $C$ denotes the area showed in Figures~\ref{Young-Geom-M-1.jpg} to~\ref{Young-Geom-C-2-1.jpg}.
\end{rem}

\subsection{Refinements of Young's integral inequality via Taylor's mean value theorem of Cauchy's type remainder and \v{C}eby\v{s}ev's integral inequality}

\begin{thm}[{\cite[Theorem~3.3]{Young-Ineq.tex}}]\label{Qi-Grad-Thm3}
Let $n\ge0$ and $h(x)\in C^{n+1}[0,c]$ such that $h(0)=0$ and $h(x)$ is strictly increasing on $[0,c]$ for $c>0$, let $h^{-1}$ be the inverse function of $h$, let $a\in[0,c]$ and $b\in[0,h(c)]$, and let $\ell\ge0$ be an integer. Then
\begin{enumerate}
\item
when
\begin{enumerate}
\item
either $h(a)>b$ and $h^{(n+1)}(x)$ is increasing on $[\alpha,\beta]$;
\item
or $h(a)<b$, $h^{(n+1)}(x)$ is increasing on $[\alpha,\beta]$, and $n=2\ell+1$;
\item
or $h(a)<b$, $h^{(n+1)}(x)$ is decreasing on $[\alpha,\beta]$, and $n=2\ell$;
\end{enumerate}
the inequality
\begin{equation}
\begin{aligned}\label{basic-eq-Young-Cebysev}
&\quad\int_0^ah(x)\td x+\int_0^bh^{-1}(x)\td x-ab\\
&\quad-\sum_{k=1}^nh^{(k)}\bigl(h^{-1}(b)\bigr)\frac{\bigl[a-h^{-1}(b)\bigr]^{k+1}}{(k+1)!}\\
&\le\frac{\bigl[a-h^{-1}(b)\bigr]^{n+1}}{(n+2)!}\bigl[h^{(n)}(a)-h^{(n)}\bigl(h^{-1}(b)\bigr)\bigr]
\end{aligned}
\end{equation}
is valid;
\item
when
\begin{enumerate}
\item
either $h(a)>b$ and $h^{(n+1)}(x)$ is decreasing on $[\alpha,\beta]$;
\item
or $h(a)<b$, $h^{(n+1)}(x)$ is increasing on $[\alpha,\beta]$, and $n=2\ell$;
\item
or $h(a)<b$, $h^{(n+1)}(x)$ is decreasing on $[\alpha,\beta]$, and $n=2\ell+1$;
\end{enumerate}
the inequality~\eqref{basic-eq-Young-Cebysev} is reversed;
\end{enumerate}
where $\alpha,\beta$ are defined as in~\eqref{alpha-beta-def-eq}.
\end{thm}

\begin{proof}
This is the outline of the proof of~\cite[Theorem~3.3]{Young-Ineq.tex}.
\par
This follows from applying the formula~\eqref{Taylor-thm-Cauchy} as in the proof of Theorem~\ref{Qi-Grad-Thm2} and applying \v{C}eby\v{s}ev's integral inequality~\eqref{Cebysev-Ineq} to the integral
\begin{equation}\label{last-term-int}
\int_{h^{-1}(b)}^a(a-t)^{n+1}h^{(n+1)}(t)\td t
\end{equation}
in the proof of Theorem~\ref{Qi-Grad-Thm2}. The proof of Theorem~\ref{Qi-Grad-Thm3} is complete.
\end{proof}

\begin{rem}
The inequality~\eqref{basic-eq-Young-Cebysev} can be geometrically interpreted as
\begin{multline*}
C-\sum_{k=1}^nh^{(k)}\bigl(h^{-1}(b)\bigr)\frac{\bigl[a-h^{-1}(b)\bigr]^{k+1}}{(k+1)!}\\
\le\frac{\bigl[a-h^{-1}(b)\bigr]^{n+1}}{(n+2)!}\bigl[h^{(n)}(a)-h^{(n)}\bigl(h^{-1}(b)\bigr)\bigr],
\end{multline*}
where $C$ denotes the area showed in Figures~\ref{Young-Geom-M-1.jpg} to~\ref{Young-Geom-C-2-1.jpg}.
\end{rem}

\subsection{Refinements of Young's integral inequality via Taylor's mean value theorem of Cauchy's type remainder and Jensen's inequalities}

\begin{thm}[{\cite[Theorem~3.4]{Young-Ineq.tex}}]\label{Qi-Grad-Thm4}
Let $h(x)\in C^{n+1}[0,c]$ such that $h(0)=0$ and $h(x)$ is strictly increasing on $[0,c]$ for $c>0$, let $h^{-1}$ be the inverse function of $h$, and let $a\in[0,c]$ and $b\in[0,h(c)]$. If $h^{(n+1)}(x)$ is convex on $[\alpha,\beta]$, where $\alpha,\beta$ are defined as in~\eqref{alpha-beta-def-eq}, then
\begin{enumerate}
\item
when $h(a)>b$ or when $h(a)<b$ and $n=2\ell$, we have
\begin{equation}
\begin{aligned}\label{basic-eq-Young-Jensen1}
&\quad\frac{\bigl[a-h^{-1}(b)\bigr]^{n+2}}{n+2}h^{(n+1)}\biggl(\frac{a+(n+2)h^{-1}(b)}{n+3}\biggr)\\
&\le\int_0^ah(x)\td x+\int_0^bh^{-1}(x)\td x-ab\\
&\quad-\sum_{k=1}^nh^{(k)}\bigl(h^{-1}(b)\bigr)\frac{\bigl[a-h^{-1}(b)\bigr]^{k+1}}{(k+1)!}\\
&\le\bigl[a-h^{-1}(b)\bigr]^{n+2}\frac{h^{(n+1)}(a)+(n+2)h^{(n+1)}\bigl(h^{-1}(b)\bigr)}{(n+3)!};
\end{aligned}
\end{equation}
\item
when $h(a)<b$ and $n=2\ell+1$, the double inequality~\eqref{basic-eq-Young-Jensen1} is reversed;
\end{enumerate}
where $\ell\ge0$ is an integer.
If $h^{(n+1)}(x)$ is concave on $[\alpha,\beta]$, all the above inequalities are reversed for all corresponding cases.
\end{thm}

\begin{proof}
This is the outline of the proof of~\cite[Theorem~3.4]{Young-Ineq.tex}.
\par
Considering the integral~\eqref{last-term-int} and substituting integral variables give
\begin{gather*}
\int_{h^{-1}(b)}^a(a-t)^{n+1}h^{(n+1)}(t)\td t
=\bigl[a-h^{-1}(b)\bigr]^{n+2}\\
\times\int_{0}^1(1-s)^{n+1}h^{(n+1)}\bigr(sa+(1-s)h^{-1}(b)\bigr)\td s.
\end{gather*}
Applying Jensen's inequalities~\eqref{Jensen-Ineq-Eq} and~\eqref{Jensen-Ineq-Int-Eq} to $h^{(n+1)}\bigr(sa+(1-s)h^{-1}(b)\bigr)$ in the above equation yield the double inequality~\eqref{basic-eq-Young-Jensen1} and its reversed version.
The proof of Theorem~\ref{Qi-Grad-Thm4} is complete.
\end{proof}

\begin{rem}
The double inequality~\eqref{basic-eq-Young-Jensen1} can be geometrically interpreted as
\begin{multline*}
\frac{\bigl[a-h^{-1}(b)\bigr]^{n+2}}{n+2}h^{(n+1)}\biggl(\frac{a+(n+2)h^{-1}(b)}{n+3}\biggr)\\
\le C-\sum_{k=1}^nh^{(k)}\bigl(h^{-1}(b)\bigr)\frac{\bigl[a-h^{-1}(b)\bigr]^{k+1}}{(k+1)!}\\
\le\bigl[a-h^{-1}(b)\bigr]^{n+2}\frac{h^{(n+1)}(a)+(n+2)h^{(n+1)}\bigl(h^{-1}(b)\bigr)}{(n+3)!},
\end{multline*}
where $C$ denotes the area showed in Figures~\ref{Young-Geom-M-1.jpg} to~\ref{Young-Geom-C-2-1.jpg}.
\end{rem}

\subsection{Refinements of Young's integral inequality via Taylor's mean value theorem of Cauchy's type remainder and integral inequalities of Hermite--Hadamard type for the product of two convex functions}

\begin{thm}[{\cite[Theorem~3.5]{Young-Ineq.tex}}]\label{Qi-Grad-Thm5}
Let $n\ge0$ and $h(x)\in C^{n+1}[0,c]$ such that $h(0)=0$ and $h(x)$ is strictly increasing on $[0,c]$ for $c>0$, let $h^{-1}$ be the inverse function of $h$, let $a\in[0,c]$ and $b\in[0,h(c)]$, and let $h^{(n+1)}(x)$ be nonnegative and convex on $[\alpha,\beta]$, where $\alpha,\beta$ are defined as in~\eqref{alpha-beta-def-eq}.
If $h(a)>b$, then
\begin{equation}
\begin{aligned}\label{Qi-Grad-Eq5}
&\quad\frac{\bigl[a-h^{-1}(b)\bigr]^{n+2}}{(n+1)!}\biggl[\frac1{2^n}h^{(n+1)}\biggl(\frac{a+h^{-1}(b)}2\biggr)\\
&\quad-\frac{2h^{(n+1)}(a)+h^{(n+1)}\bigl(h^{-1}(b)\bigr)}{6}\biggr]\\
&\le\int_0^ah(x)\td x+\int_0^bh^{-1}(x)\td x-ab\\
&\quad-\sum_{k=1}^nh^{(k)}\bigl(h^{-1}(b)\bigr)\frac{\bigl[a-h^{-1}(b)\bigr]^{k+1}}{(k+1)!}\\
&\le\frac{\bigl[a-h^{-1}(b)\bigr]^{n+2}}{(n+1)!}\frac{h^{(n+1)}(a)+2h^{(n+1)}\bigl(h^{-1}(b)\bigr)}{6}.
\end{aligned}
\end{equation}
If $h(a)<b$ and $n=2\ell$ for $\ell\ge0$, then
\begin{equation}
\begin{aligned}\label{Qi-Grad-Eq6}
&\quad\frac{\bigl[h^{-1}(b)-a\bigr]^{n+2}}{(n+1)!}\biggl[\frac1{2^n}h^{(n+1)}\biggl(\frac{a+h^{-1}(b)}2\biggr)\\
&\quad-\frac{2h^{(n+1)}(a)+h^{(n+1)}\bigl(h^{-1}(b)\bigr)}{6}\biggr]\\
&\le\int_0^ah(x)\td x+\int_0^bh^{-1}(x)\td x-ab\\
&\quad-\sum_{k=1}^nh^{(k)}\bigl(h^{-1}(b)\bigr) \frac{\bigl[a-h^{-1}(b)\bigr]^{k+1}}{(k+1)!}\\
&\le\frac{\bigl[h^{-1}(b)-a\bigr]^{n+2}}{(n+1)!}\frac{h^{(n+1)}(a)+2h^{(n+1)}\bigl(h^{-1}(b)\bigr)}{6}.
\end{aligned}
\end{equation}
If $a<h^{-1}(b)$ and $n=2\ell+1$ for $\ell\ge0$, the double inequality~\eqref{Qi-Grad-Eq6} is reversed.
\end{thm}

\begin{proof}
This is the outline of the proof of~\cite[Theorem~3.5]{Young-Ineq.tex}.
\par
Let $f(x)$ and $g(x)$ be nonnegative and convex functions on $[\mu,\nu]$. Then
\begin{multline}\label{Pachpatte-RGMIA-product}
2f\biggl(\frac{\mu+\nu}2\biggr)g\biggl(\frac{\mu+\nu}2\biggr)-\frac{1}{6}M(\mu,\nu)-\frac13N(\mu,\nu)\\
\le\frac{1}{\nu-\mu}\int_{\mu}^{\nu}f(x)g(x)\td x
\le\frac{1}{3}M(\mu,\nu)+\frac16N(\mu,\nu),
\end{multline}
where
\begin{equation*}
M(\mu,\nu)=f(\mu)g(\mu)+f(\nu)g(\nu) \quad\text{and}\quad N(\mu,\nu)=f(\mu)g(\nu)+f(\nu)g(\mu).
\end{equation*}
The double inequality~\eqref{Pachpatte-RGMIA-product} can be found in~\cite{Pachpatte-RGMIA-03, Wu-Qi-Maejo.tex, Hong-Ping-Qi-JNSA.tex, H-H-(am)-YHP-Missouri.tex} and closely related references therein. Applying~\eqref{Pachpatte-RGMIA-product} in the integral~\eqref{last-term-int}
arrives at the double inequalities in~\eqref{Qi-Grad-Eq5} and~\eqref{Qi-Grad-Eq6}. The proof of Theorem~\ref{Qi-Grad-Thm5} is complete.
\end{proof}

\begin{rem}
The double inequalities~\eqref{Qi-Grad-Eq5} and~\eqref{Qi-Grad-Eq6} can be geometrically interpreted as
\begin{multline*}
\frac{\bigl[a-h^{-1}(b)\bigr]^{n+2}}{(n+1)!}\biggl[\frac1{2^n}h^{(n+1)}\biggl(\frac{a+h^{-1}(b)}2\biggr)
-\frac{2h^{(n+1)}(a)+h^{(n+1)}\bigl(h^{-1}(b)\bigr)}{6}\biggr]\\
\le C-\sum_{k=1}^nh^{(k)}\bigl(h^{-1}(b)\bigr)\frac{\bigl[a-h^{-1}(b)\bigr]^{k+1}}{(k+1)!}\\
\le\frac{\bigl[a-h^{-1}(b)\bigr]^{n+2}}{(n+1)!}\frac{h^{(n+1)}(a)+2h^{(n+1)}\bigl(h^{-1}(b)\bigr)}{6}
\end{multline*}
and
\begin{multline*}
\frac{\bigl[h^{-1}(b)-a\bigr]^{n+2}}{(n+1)!}\biggl[\frac1{2^n}h^{(n+1)}\biggl(\frac{a+h^{-1}(b)}2\biggr)
-\frac{2h^{(n+1)}(a)+h^{(n+1)}\bigl(h^{-1}(b)\bigr)}{6}\biggr]\\
\le C-\sum_{k=1}^nh^{(k)}\bigl(h^{-1}(b)\bigr) \frac{\bigl[a-h^{-1}(b)\bigr]^{k+1}}{(k+1)!}\\
\le\frac{\bigl[h^{-1}(b)-a\bigr]^{n+2}}{(n+1)!}\frac{h^{(n+1)}(a)+2h^{(n+1)}\bigl(h^{-1}(b)\bigr)}{6},
\end{multline*}
where $C$ denotes the area showed in Figures~\ref{Young-Geom-M-1.jpg} to~\ref{Young-Geom-C-2-1.jpg}.
\end{rem}

\subsection{Three examples showing refinements of Young's integral inequality}\label{sec-three-examp}

\subsubsection{First example}
In~\cite[Section~3]{Young-Hoorfar-Qi.tex}, the double inequality~\eqref{Hoofar-Qi-eq4} was applied to obtain the estimate
\begin{align*}
9.000042866\dotsc
&=\frac{4\sqrt[4]{125}\,}{27}\bigl(3-2\sqrt[4]{5}\,\bigr)^2\\
&<\int^{3}_{0}\sqrt[4]{x^{4}+1}\,\td x +\int^{3}_{1}\sqrt[4]{x^{4}-1}\,\td x-9\\
&<\frac{27}{2\sqrt[4]{82^3}\,}\bigl(3-2\sqrt[4]{5}\,\bigr)^2\\
&=9.000042871\dotsc
\end{align*}
whose gap between the upper and lower bounds is $0.000000005\dotsc$ and which refines a known result
\begin{equation*}
9<\int^{3}_{0}\sqrt[4]{x^{4}+1}\,\td x +\int^{3}_{1}\sqrt[4]{x^{4}-1}\,\td x<9.0001
\end{equation*}
\par
In~\cite[Example~2.5]{MIA-2010-13-43} and~\cite[Remark~2.7]{MIA-2010-13-43}, it was obtained that
\begin{equation*}
9.000042866<\int^{3}_{0}\sqrt[4]{x^{4}+1}\,\td x +\int^{3}_{1}\sqrt[4]{x^{4}-1}\,\td x<9.000042868880
\end{equation*}
and
\begin{equation*}
9.000042868058<\int^{3}_{0}\sqrt[4]{x^{4}+1}\,\td x +\int^{3}_{1}\sqrt[4]{x^{4}-1}\,\td x<9.000042868066.
\end{equation*}
whose gaps between the upper and lower bounds are
\begin{equation*}
0.0000000028\dotsc\quad \text{and}\quad 0.000000000008\dotsc
\end{equation*}
respectively.
\par
In~\cite[Example~2.1]{Jak-Pecaric-2009-AEJM}, it was estimated that
\begin{equation*}
9.00004286765564<\int^{3}_{0}\sqrt[4]{x^{4}+1}\,\td x +\int^{3}_{1}\sqrt[4]{x^{4}-1}\,\td x<9.00004286805781.
\end{equation*}
whose gap between the upper and lower bounds is $0.0000000004021\dotsc$.
\par
In~\cite[Example~4.1]{Young-Ineq.tex}, by virtue of the double inequality~\eqref{basic-eq-Young-Jensen1}, the above double inequality was refined as
\begin{align*}
9.0000428983186013\dotsc
&=\frac{\bigl(3-80^{1/4}\bigr)^3}{3}\frac{3072\bigl(\sqrt[4]{95}\, \sqrt{2}\,+3\bigr)^2}{\bigl[\bigl(\sqrt[4]{95}\,\sqrt{2}\,+3\bigr)^4+256\bigr]^{7/4}}\\
&\quad+9+\frac{8\times5^{3/4}}{27}\frac{\bigl(3-80^{1/4}\bigr)^2}{2!}\\
&\ge\int^{3}_{0}\sqrt[4]{x^{4}+1}\,\td x+\int^{3}_{1}\sqrt[4]{x^{4}-1}\,\td x\\
&\ge\frac{\bigl(3-80^{1/4}\bigr)^3}{4!}\biggl(\frac{27}{82^{7/4}}+3\times\frac{4\sqrt{5}\,}{729}\biggr)\\
&\quad+9+\frac{8\times5^{3/4}}{27}\frac{\bigl(3-80^{1/4}\bigr)^2}{2!}\\
&=9.0000428680640760\dotsc
\end{align*}
whose gap between the upper and lower bounds is $0.00000003025452\dotsc$.

\subsubsection{Second example}
In~\cite[Example~4.2]{Young-Ineq.tex}, by virtue of the double inequality~\eqref{theorem2.3-mia}, it follows that
\begin{equation}
\begin{aligned}\label{exp-recip-bounds}
0.364469045537996606\dotsc
&=\frac14+\biggl(\frac12-\frac1{\ln2}\biggr)\biggl[\frac1{\exp[\frac12(\frac12+\frac1{\ln2})]}-\frac12\biggr]\\
&\le\int_0^{1/2}\frac1{e^{1/x}}\td x-\int_0^{1/2}\frac1{\ln x}\td x\\
&\le\frac14+\frac12\biggl(\frac12-\frac1{\ln2}\biggr)\biggl(\frac1{e^2}-\frac12\biggr)\\
&=0.421883810040011829\dotsc.
\end{aligned}
\end{equation}
The gap between the upper and lower bounds in the double inequality~\eqref{exp-recip-bounds} is $0.057414764502015\dotsc$.

\subsubsection{Third example}
In~\cite[Example~4.3]{Young-Ineq.tex}, by virtue of the double inequality~\eqref{Qi-Grad-Eq5}, we can obtain the estimate
\begin{equation}\label{exp-log-sqrt-ineq}
2.044751320\dotsc\le\int_0^1e^{t^2}\td t+\int_0^1\sqrt{\ln(1+t)}\,\td t\le2.060536019\dotsc.
\end{equation}
The gap between the upper and lower bounds in the double inequality~\eqref{exp-log-sqrt-ineq} is $0.01578469\dotsc$.

\section{New refinements of Young's integral inequality via P\'olya's type integral inequalities}

In this section, by virtue of P\'olya's type integral inequalities~\cite{Polya-MIA-3829.tex, S-Y-Q-n-D.tex}, we establish some new refinements in terms of higher order derivatives.

\subsection{Refinements of Young's integral inequality in terms of bounds of the first derivative}

\begin{thm}\label{polya-young-thm}
Let $h(x)$ be a strictly increasing function on $[0,c]$ for $c>0$ and let $h^{-1}$ be the inverse function of $h$. If $h(0)=0$, $a\in[0,c]$, $b\in[0,h(c)]$, $L$ and $U$ are real constants, and $L\le h'(x)\le U$ on $(\alpha,\beta)$, then
\begin{equation}
\begin{aligned}\label{polya-young-ineq}
&\quad\frac{LU\bigl[a-h^{-1}(b)\bigr]^2-2\bigl[a-h^{-1}(b)\bigr][Lh(a)-Ub] +[h(a)-b]^2}{2(U-L)}\\
&\le\int_0^ah(x)\td x+\int_0^bh^{-1}(x)\td x-bh^{-1}(b)\\
&\le-\frac{LU\bigl[a-h^{-1}(b)\bigr]^2-2\bigl[a-h^{-1}(b)\bigr] [Uh(a)-Lb] +[h(a)-b]^2}{2(U-L)}.
\end{aligned}
\end{equation}
\end{thm}

\begin{proof}
Let $f(x)$ be continuous on $[a, b]$ and differentiable on $(a,b)$. If $f(x)$ is not identically a constant and $m\le f'(x)\le M$ in
$(a,b)$, then
\begin{multline}\label{4.466}
\biggl|\frac1{b-a}\int_a^bf(x)\td x-\frac{f(a)+f(b)}2\biggr|\\
\le\frac{(M-m)(b-a)}2\Biggl[\frac14-\frac{\bigl(\frac{f(b)-f(a)}
{b-a}-\frac{M+m}2\bigr)^2}{(M-m)^2}\Biggr].
\end{multline}
The inequality~\eqref{4.466} can be rearranged as a double inequality
\begin{multline}\label{qizhcor14}
\frac{mM(b-a)^2-2(b-a)[mf(b)-Mf(a)]+[f(b)-f(a)]^2}{2(M-m)}\le\int_a^bf(x)\td x\\
\le-\frac{mM(b-a)^2-2(b-a)[Mf(b)-mf(a)]+[f(b)-f(a)]^2}{2(M-m)}.
\end{multline}
These inequalities can be found in~\cite[Theorem~2]{rpss}, the papers~\cite{pcerone, Cerone-Hacet-2015, cer1, Cerone-Dragomir-MIA2000, Cheng-AML-2001}, \cite[Proposition~2]{gaz}, \cite[Section~5]{Polya-MIA-3829.tex} and closely related references therein.
\par
The area $C$ can be computed by~\eqref{C-area-exp-unif} in Remark~\ref{C-area-exp-unif-rem}, which can be estimated, by applying the double inequality~\eqref{qizhcor14}, as
\begin{multline*}
\frac{LU\bigl[a-h^{-1}(b)\bigr]^2-2\bigl[a-h^{-1}(b)\bigr][Lh(a)-Ub] +[h(a)-b]^2}{2(U-L)}
\le\int_{h^{-1}(b)}^{a}h(x)\td x\\
\le-\frac{LU\bigl[a-h^{-1}(b)\bigr]^2-2\bigl[a-h^{-1}(b)\bigr] [Uh(a)-Lb] +[h(a)-b]^2}{2(U-L)}.
\end{multline*}
Since
\begin{equation*}
\int_{h^{-1}(b)}^{a}h(x)\td x-b\bigl[a-h^{-1}(b)\bigr]=\int_0^ah(x)\td x+\int_0^bh^{-1}(x)\td x-ab,
\end{equation*}
that is,
\begin{equation}\label{3integrals-eq}
\int_{h^{-1}(b)}^{a}h(x)\td x=\int_0^ah(x)\td x+\int_0^bh^{-1}(x)\td x-bh^{-1}(b),
\end{equation}
the double inequality~\eqref{polya-young-ineq} follows straightforwardly.
The proof of Theorem~\ref{polya-young-thm} is complete.
\end{proof}

\subsection{Refinements of Young's integral inequality in terms of bounds of the second derivative}

\begin{thm}\label{second-der-thm}
Let $h(x)$ be a strictly increasing function on $[0,c]$ for $c>0$, let $h^{-1}$ be the inverse function of $h$, let $h(0)=0$, $a\in[0,c]$, and $b\in[0,h(c)]$, and let $L$ and $U$ be real constants such that $L\le h''(x)\le U$ on $(\alpha,\beta)$. Then
\begin{multline}\label{second-deriv-H-Y}
\frac{L\bigl[a^3-\bigl(h^{-1}(b)\bigr)^3\bigr]}{6} +\frac{\Biggl(\begin{gathered}b-h(a)+ah'(a)-h^{-1}(b)h'\bigl(h^{-1}(b)\bigr)\\ +L\bigl[\bigl(h^{-1}(b)\bigr)^2-a^2\bigr]/2\end{gathered}\Biggr)^2} {2\bigl[\bigl(h^{-1}(b)-a\bigr)L-h'\bigl(h^{-1}(b)\bigr)+h'(a)\bigr]}\\
\le\int_0^ah(x)\td x+\int_0^bh^{-1}(x)\td x -ah(a)+\frac{a^2h'(a)-\bigl[h^{-1}(b)\bigr]^2h'\bigl(h^{-1}(b)\bigr)}{2}\\
\le\frac{U\bigl[a^3-\bigl(h^{-1}(b)\bigr)^3\bigr]}{6}+\frac{\Biggl(\begin{gathered}b-h(a)+ah'(a)-h^{-1}(b)h'\bigl(h^{-1}(b)\bigr)\\
+U\bigl[\bigl(h^{-1}(b)\bigr)^2-a^2\bigr]/2\end{gathered}\Biggr)^2} {2\bigl[\bigl(h^{-1}(b)-a\bigr)U-h'\bigl(h^{-1}(b)\bigr)+h'(a)\bigr]}.
\end{multline}
\end{thm}

\begin{proof}
In~\cite[Corallary]{YUGOS.TEX} and~\cite[Corallary~1.2]{ineq-weighted-integral.tex}, it was acquired that, if $f(x)\in C([a,b])$ satisfying $N\le f''(x)\le M$ on $(a,b)$, then
\begin{multline}\label{second-deriv-23}
\frac{N(b^3-a^3)}{6} +\frac{\bigl[f(a)-f(b)+bf'(b)-af'(a)+N\bigl(a^2-b^2\bigr)/2\bigr]^2} {2[(a-b)N-f'(a)+f'(b)]}\\
 \le\int_a^bf(x)\td x -bf(b)+af(a)+\frac{b^2f'(b)-a^2f'(a)}{2}\\
 \le\frac{M(b^3-a^3)}{6}+\frac{\bigl[f(a)-f(b)+bf'(b)-af'(a) +M\bigl(a^2-b^2\bigr)/2\bigr]^2} {2[(a-b)M-f'(a)+f'(b)]}.
\end{multline}
Applying the double inequality~\eqref{second-deriv-23} to the integral $\int_{h^{-1}(b)}^{a}h(x)\td x$ and considering Remark~\ref{C-area-exp-unif-rem} yield
\begin{multline}\label{second-deriv-H}
\frac{L\bigl[a^3-\bigl(h^{-1}(b)\bigr)^3\bigr]}{6} +\frac{\Biggl(\begin{gathered}b-h(a)+ah'(a)-h^{-1}(b)h'\bigl(h^{-1}(b)\bigr)\\ +L\bigl[\bigl(h^{-1}(b)\bigr)^2-a^2\bigr]/2\end{gathered}\Biggr)^2} {2\bigl[\bigl(h^{-1}(b)-a\bigr)L-h'\bigl(h^{-1}(b)\bigr)+h'(a)\bigr]}\\
\le\int_{h^{-1}(b)}^ah(x)\td x -ah(a)+bh^{-1}(b)+\frac{a^2h'(a)-\bigl[h^{-1}(b)\bigr]^2h'\bigl(h^{-1}(b)\bigr)}{2}\\
\le\frac{U\bigl[a^3-\bigl(h^{-1}(b)\bigr)^3\bigr]}{6}+\frac{\Biggl(\begin{gathered}b-h(a)+ah'(a)-h^{-1}(b)h'\bigl(h^{-1}(b)\bigr)\\
+U\bigl[\bigl(h^{-1}(b)\bigr)^2-a^2\bigr]/2\end{gathered}\Biggr)^2} {2\bigl[\bigl(h^{-1}(b)-a\bigr)U-h'\bigl(h^{-1}(b)\bigr)+h'(a)\bigr]}.
\end{multline}
Substituting~\eqref{3integrals-eq} into~\eqref{second-deriv-H} results in the double inequality~\eqref{second-deriv-H-Y}.
The proof of Theorem~\ref{second-der-thm} is complete.
\end{proof}

\subsection{Refinements of Young's integral inequality in terms of bounds of higher order derivatives}

\begin{thm}\label{further=beog-polya}
Let $h(x)$ be a strictly increasing function on $[0,c]$ for $c>0$, let $h^{-1}$ be the inverse of $h$, let $h(0)=0$, $a\in[0,c]$, and $b\in[0,h(c)]$, and let $h(x)$ have the $(n+1)$-th derivative on $[0,c]$ such that $L\le h^{(n+1)}(x)\le U$ on $(\alpha,\beta)$. Then, for all $t$ between $a$ and $h^{-1}(b)$,
\begin{enumerate}
\item
when $n$ is a nonnegative odd integer,
\begin{multline}\label{R-beograd21}
\sum_{i=0}^{n+2}\frac{(-1)^i}{i!}\Bigl[S_{n+2}^{(i)}\bigl(h;h^{-1}(b),h^{-1}(b),L\bigr)-S_{n+2}^{(i)}(h;a,a,L)\Bigr]t^i\\
\le\int_0^ah(x)\td x+\int_0^bh^{-1}(x)\td x-bh^{-1}(b)\\
\le \sum_{i=0}^{n+2}\frac{(-1)^i}{i!}\Bigl[S_{n+2}^{(i)}\bigl(h;h^{-1}(b),h^{-1}(b),U\bigr)-S_{n+2}^{(i)}(h;a,a,U)\Bigr]t^i;
\end{multline}
\item
when $n$ is a nonnegative even integer,
\begin{multline}\label{R-beograd22}
\sum_{i=0}^{n+2}\frac{(-1)^i}{i!}\Bigl[S_{n+2}^{(i)}\bigl(h;h^{-1}(b),h^{-1}(b),L\bigr)-S_{n+2}^{(i)}(h;a,a,U)\Bigr]t^i\\
\le\int_0^ah(x)\td x+\int_0^bh^{-1}(x)\td x-bh^{-1}(b)\\
\le \sum_{i=0}^{n+2}\frac{(-1)^i}{i!}\Bigl[S_{n+2}^{(i)}\bigl(h;h^{-1}(b),h^{-1}(b),U\bigr)-S_{n+2}^{(i)}(h;a,a,L)\Bigr]t^i;
\end{multline}
\end{enumerate}
where $L$ and $U$ are real constants,
\begin{equation*}
S_n(h;u,v,w)=\sum_{k=1}^{n-1}\frac{(-1)^k}{k!}u^kh^{(k-1)}(v)+(-1)^n\frac{w}{n!}u^n,
\end{equation*}
and
\begin{equation*}
S_n^{(k)}(h;u,v,w)=\frac{\partial^kS_n(h;u,v,w)}{\partial u^k}.
\end{equation*}
\end{thm}

\begin{proof}
In~\cite[Theorem]{YUGOS.TEX}, it was discovered that, if $f\in C^n([a,b])$ has derivative of $(n+1)$-th order satisfying $N\le f^{(n+1)}(x)\le M$ on $(a,b)$, then, for all $t\in (a,b)$,
\begin{enumerate}
\item
when $n$ is a nonnegative odd integer,
\begin{multline}\label{beograd21}
\sum_{i=0}^{n+2}\frac{(-1)^i}{i!}\Bigl[S_{n+2}^{(i)}(f;a,a,N)-S_{n+2}^{(i)}(f;b,b,N)\Bigr]t^i\le
\int_a^bf(x)\td x \\
\le \sum_{i=0}^{n+2}\frac{(-1)^i}{i!}\Bigl[S_{n+2}^{(i)}(f;a,a,M)-S_{n+2}^{(i)}(f;b,b,M)\Bigr]t^i;
\end{multline}
\item
when $n$ is a nonnegative even integer,
\begin{multline}\label{beograd22}
\sum_{i=0}^{n+2}\frac{(-1)^i}{i!}\Bigl[S_{n+2}^{(i)}(f;a,a,N)-S_{n+2}^{(i)}(f;b,b,M)\Bigr]t^i\le\int_a^bf(x)\td x \\
\le \sum_{i=0}^{n+2}\frac{(-1)^i}{i!}\Bigl[S_{n+2}^{(i)}(f;a,a,M)-S_{n+2}^{(i)}(f;b,b,N)\Bigr]t^i.
\end{multline}
\end{enumerate}
These inequalities can also be found in~\cite{hung-acta-99, Mult-Qi-JMAA, Some-New-Iyengar-Type-Inequalities.tex, difference-Hermite--Hadamard.tex, S-Y-Q-n-D.tex} and closely related references therein.
\par
Applying~\eqref{beograd21} and~\eqref{beograd22} to the integral $\int_{h^{-1}(b)}^{a}h(x)\td x$ and considering Remark~\ref{C-area-exp-unif-rem} yield
\begin{enumerate}
\item
when $n$ is a nonnegative odd integer,
\begin{multline}\label{h-beograd21}
\sum_{i=0}^{n+2}\frac{(-1)^i}{i!}\Bigl[S_{n+2}^{(i)}\bigl(h;h^{-1}(b),h^{-1}(b),L\bigr)-S_{n+2}^{(i)}(h;a,a,L)\Bigr]t^i
\le\int_{h^{-1}(b)}^{a}h(x)\td x\\
\le \sum_{i=0}^{n+2}\frac{(-1)^i}{i!}\Bigl[S_{n+2}^{(i)}\bigl(h;h^{-1}(b),h^{-1}(b),U\bigr)-S_{n+2}^{(i)}(h;a,a,U)\Bigr]t^i;
\end{multline}
\item
when $n$ is a nonnegative even integer,
\begin{multline}\label{h-beograd22}
\sum_{i=0}^{n+2}\frac{(-1)^i}{i!}\Bigl[S_{n+2}^{(i)}\bigl(h;h^{-1}(b),h^{-1}(b),L\bigr)-S_{n+2}^{(i)}(h;a,a,U)\Bigr]t^i
\le\int_{h^{-1}(b)}^{a}h(x)\td x\\
\le \sum_{i=0}^{n+2}\frac{(-1)^i}{i!}\Bigl[S_{n+2}^{(i)}\bigl(h;h^{-1}(b),h^{-1}(b),U\bigr)-S_{n+2}^{(i)}(h;a,a,L)\Bigr]t^i.
\end{multline}
\end{enumerate}
Substituting~\eqref{3integrals-eq} into~\eqref{h-beograd21} and~\eqref{h-beograd22} results in~\eqref{R-beograd21} and~\eqref{R-beograd22}.
The proof of Theorem~\ref{further=beog-polya} is complete.
\end{proof}

\subsection{Refinements of Young's integral inequality in terms of $L^p$-norms}

\begin{thm}\label{norm=beog-polya}
Let $h(x)$ be a strictly increasing function on $[0,c]$ for $c>0$, let $h^{-1}$ be the inverse of $h$, let $h(0)=0$, $a\in[0,c]$, and $b\in[0,h(c)]$, and let $h(x)$ have the $(n+1)$-th derivative on $[\alpha,\beta]$ such that $h^{(n+1)}\in L^p([\alpha,\beta])$ for $p,q>0$ with $\frac1p+\frac1q=1$. Then, for all $t\in[\alpha,\beta]$,
\begin{enumerate}
\item
when $p,q>1$, we have
\begin{multline}\label{thm-est1}
\Biggl|\int_{h^{-1}(b)}^ah(x)\td x-\sum_{i=0}^{n}\frac{h^{(i)}\bigl(h^{-1}(b)\bigr)}{(i+1)!}\bigl(t-h^{-1}(b)\bigr)^{i+1} +\sum_{i=0}^{n}\frac{h^{(i)}(a)}{(i+1)!}(t-a)^{i+1}\Biggr|\\
\begin{aligned}
&\le\frac{\bigl(t-h^{-1}(b)\bigr)^{n+1+1/q}+(a-t)^{n+1+1/q}}{(n+1)!\sqrt[q]{nq+q+1}}\bigl\|h^{(n+1)}\bigr\|_{L^p([h^{-1}(b),a])}\\
&\le\frac{2\bigl(a-h^{-1}(b)\bigr)^{n+2}}{(n+1)!}\bigl\|h^{(n+1)}\bigr\|_{L^p([h^{-1}(b),a])};
\end{aligned}
\end{multline}
\item
when $p=\infty$, we have
\begin{multline}\label{thm-est2}
\Biggl|\int_{h^{-1}(b)}^ah(x)\td x-\sum_{i=0}^{n}\frac{h^{(i)}\bigl(h^{-1}(b)\bigr)}{(i+1)!}\bigl(t-h^{-1}(b)\bigr)^{i+1} +\sum_{i=0}^{n}\frac{h^{(i)}(a)}{(i+1)!}(t-a)^{i+1}\Biggr|\\
\begin{aligned}
&\le\frac{\bigl(t-h^{-1}(b)\bigr)^{n+2}+(a-t)^{n+2}}{(n+2)!}\bigl\|h^{(n+1)}\bigr\|_{L^\infty([h^{-1}(b),a])}\\
&\le\frac{2\bigl(a-h^{-1}(b)\bigr)^{n+2}}{(n+2)!}\bigl\|h^{(n+1)}\bigr\|_{L^\infty([h^{-1}(b),a])};
\end{aligned}
\end{multline}
\item
when $p=1$, we have
\begin{multline}\label{thm-est3}
\Biggl|\int_{h^{-1}(b)}^ah(x)\td x-\sum_{i=0}^{n}\frac{h^{(i)}\bigl(h^{-1}(b)\bigr)}{(i+1)!}\bigl(t-h^{-1}(b)\bigr)^{i+1} +\sum_{i=0}^{n}\frac{h^{(i)}(a)}{(i+1)!}(t-a)^{i+1}\Biggr|\\
\begin{aligned}
&\le\frac{\bigl(t-h^{-1}(b)\bigr)^{n+1}+(a-t)^{n+1}}{(n+1)!}\bigl\|h^{(n+1)}\bigr\|_{L([h^{-1}(b),a])}\\
&\le\frac{2\bigl(a-h^{-1}(b)\bigr)^{n+1}}{(n+1)!}\bigl\|h^{(n+1)}\bigr\|_{L([h^{-1}(b),a])}.
\end{aligned}
\end{multline}
\end{enumerate}
\end{thm}

\begin{proof}
Let $f\in C^n([a,b])$ have derivative of $(n+1)$-th order on $(a,b)$ and $f^{(n+1)}\in L^p([a,b])$ for positive numbers $p$ and $q$ satisfying $\frac1p+\frac1q=1$. In~\cite{lpita-Nova-03} and~\cite[Theorem~2]{Analysis_Mathematica-L_p.tex}, it was established that, for any $t\in(a,b)$,
\begin{enumerate}
\item
when $p,q>1$, we have
\begin{multline}\label{est1}
\Biggl|\int_a^bf(x)\td x-\sum_{i=0}^{n}\frac{f^{(i)}(a)}{(i+1)!}(t-a)^{i+1} +\sum_{i=0}^{n}\frac{f^{(i)}(b)}{(i+1)!}(t-b)^{i+1}\Biggr|\\
\begin{aligned}
&\le\frac{(t-a)^{n+1+1/q}+(b-t)^{n+1+1/q}}{(n+1)!\sqrt[q]{nq+q+1}}\bigl\|f^{(n+1)}\bigr\|_{L^p([a,b])}\\
&\le\frac{2(b-a)^{n+2}}{(n+1)!}\bigl\|f^{(n+1)}\bigr\|_{L^p([a,b])};
\end{aligned}
\end{multline}
\item
when $p=\infty$, we have
\begin{multline}\label{est2}
\Biggl|\int_a^bf(x)\td x-\sum_{i=0}^{n}\frac{f^{(i)}(a)}{(i+1)!}(t-a)^{i+1} +\sum_{i=0}^{n}\frac{f^{(i)}(b)}{(i+1)!}(t-b)^{i+1}\Biggr|\\
\begin{aligned}
&\le\frac{(t-a)^{n+2}+(b-t)^{n+2}}{(n+2)!}\bigl\|f^{(n+1)}\bigr\|_{L^\infty([a,b])}\\
&\le\frac{2(b-a)^{n+2}}{(n+2)!}\bigl\|f^{(n+1)}\bigr\|_{L^\infty([a,b])};
\end{aligned}
\end{multline}
\item
when $p=1$, we have
\begin{multline}\label{est3}
\Biggl|\int_a^bf(x)\td x-\sum_{i=0}^{n}\frac{f^{(i)}(a)}{(i+1)!}(t-a)^{i+1} +\sum_{i=0}^{n}\frac{f^{(i)}(b)}{(i+1)!}(t-b)^{i+1}\Biggr|\\
\begin{aligned}
&\le\frac{(t-a)^{n+1}+(b-t)^{n+1}}{(n+1)!}\bigl\|f^{(n+1)}\bigr\|_{L([a,b])}\\
&\le\frac{2(b-a)^{n+1}}{(n+1)!}\bigl\|f^{(n+1)}\bigr\|_{L([a,b])}.
\end{aligned}
\end{multline}
\end{enumerate}
Applying three inequalities~\eqref{est1}, \eqref{est2}, and~\eqref{est3} to the integral $\int_{h^{-1}(b)}^{a}h(x)\td x$ and considering Remark~\ref{C-area-exp-unif-rem} lead to the following conclusions:
\begin{enumerate}
\item
when $p,q>1$, we have
\begin{multline}\label{h-est1}
\Biggl|\int_{h^{-1}(b)}^ah(x)\td x-\sum_{i=0}^{n}\frac{h^{(i)}\bigl(h^{-1}(b)\bigr)}{(i+1)!}\bigl(t-h^{-1}(b)\bigr)^{i+1} +\sum_{i=0}^{n}\frac{h^{(i)}(a)}{(i+1)!}(t-a)^{i+1}\Biggr|\\
\begin{aligned}
&\le\frac{\bigl(t-h^{-1}(b)\bigr)^{n+1+1/q}+(a-t)^{n+1+1/q}}{(n+1)!\sqrt[q]{nq+q+1}}\bigl\|h^{(n+1)}\bigr\|_{L^p([h^{-1}(b),a])}\\
&\le\frac{2\bigl(a-h^{-1}(b)\bigr)^{n+2}}{(n+1)!}\bigl\|h^{(n+1)}\bigr\|_{L^p([h^{-1}(b),a])};
\end{aligned}
\end{multline}
\item
when $p=\infty$, we have
\begin{multline}\label{h-est2}
\Biggl|\int_{h^{-1}(b)}^ah(x)\td x-\sum_{i=0}^{n}\frac{h^{(i)}\bigl(h^{-1}(b)\bigr)}{(i+1)!}\bigl(t-h^{-1}(b)\bigr)^{i+1} +\sum_{i=0}^{n}\frac{h^{(i)}(a)}{(i+1)!}(t-a)^{i+1}\Biggr|\\
\begin{aligned}
&\le\frac{\bigl(t-h^{-1}(b)\bigr)^{n+2}+(a-t)^{n+2}}{(n+2)!}\bigl\|h^{(n+1)}\bigr\|_{L^\infty([h^{-1}(b),a])}\\
&\le\frac{2\bigl(a-h^{-1}(b)\bigr)^{n+2}}{(n+2)!}\bigl\|h^{(n+1)}\bigr\|_{L^\infty([h^{-1}(b),a])};
\end{aligned}
\end{multline}
\item
when $p=1$, we have
\begin{multline}\label{h-est3}
\Biggl|\int_{h^{-1}(b)}^ah(x)\td x-\sum_{i=0}^{n}\frac{h^{(i)}\bigl(h^{-1}(b)\bigr)}{(i+1)!}\bigl(t-h^{-1}(b)\bigr)^{i+1} +\sum_{i=0}^{n}\frac{h^{(i)}(a)}{(i+1)!}(t-a)^{i+1}\Biggr|\\
\begin{aligned}
&\le\frac{\bigl(t-h^{-1}(b)\bigr)^{n+1}+(a-t)^{n+1}}{(n+1)!}\bigl\|h^{(n+1)}\bigr\|_{L([h^{-1}(b),a])}\\
&\le\frac{2\bigl(a-h^{-1}(b)\bigr)^{n+1}}{(n+1)!}\bigl\|h^{(n+1)}\bigr\|_{L([h^{-1}(b),a])}.
\end{aligned}
\end{multline}
\end{enumerate}
Substituting~\eqref{3integrals-eq} into~\eqref{h-est1}, \eqref{h-est2}, and~\eqref{h-est3} results in~\eqref{thm-est1}, \eqref{thm-est2}, and~\eqref{thm-est3}.
The proof of Theorem~\ref{norm=beog-polya} is complete.
\end{proof}

\subsection{Three examples for new refinements of Young's integral inequalities}

\subsubsection{First example}
Let $h(x)=\sqrt[4]{x^{4}+1}\,-1$ and let $a=3$ and $b=2$ in Theorem~\ref{polya-young-thm}. Then
\begin{gather*}
h'(x)=\frac{x^3}{(x^4+1)^{3/4}}, \quad h''(x)=\frac{3 x^2}{(x^4+1)^{7/4}}>0,\\
h^{-1}(2)=\bigr(3^4-1\bigr)^{1/4}=2\sqrt[4]{5}\,=2.990\dotsc,\\
L=h'\bigl(2\sqrt[4]{5}\,\bigr)=\frac{8\times5^{3/4}}{27}, \quad U=h'(3)=\frac{27}{82^{3/4}},\\
h(3)=\sqrt[4]{82}\,-1, \quad \int_0^3h(x)\td x=\int_0^3\sqrt[4]{x^{4}+1}\,\td x-3,\\
\int_0^2h^{-1}(x)\td x=\int_{0}^{2}\sqrt[4]{(x+1)^4-1}\,=\int_{1}^{3}\sqrt[4]{x^4-1}\td x,
\end{gather*}
and
\begin{gather*}
\frac{\left(\begin{gathered}\frac{8\times5^{3/4}}{27}\frac{27}{82^{3/4}} \bigl(3-2\sqrt[4]{5}\,\bigr)^2 +\bigl(\sqrt[4]{82}\,-3\bigr)^2\\
-2\bigl(3-2\sqrt[4]{5}\,\bigr)\biggl[\frac{8\times5^{3/4}}{27} \bigl(\sqrt[4]{82}\,-1\bigr)-\frac{2\times27}{82^{3/4}}\biggr]\end{gathered}\right)} {2\bigl(\frac{27}{82^{3/4}}-\frac{8\times5^{3/4}}{27}\bigr)}\\
\le\int_0^3\sqrt[4]{x^{4}+1}\,\td x+\int_{1}^{3}\sqrt[4]{x^4-1}\td x-3-4\sqrt[4]{5}\,\\
\le-\frac{\left(\begin{gathered}\frac{8\times5^{3/4}}{27}\frac{27}{82^{3/4}} \bigl(3-2\sqrt[4]{5}\,\bigr)^2 +\bigl(\sqrt[4]{82}\,-3\bigr)^2\\
-2\bigl(3-2\sqrt[4]{5}\,\bigr) \biggl[\frac{27}{82^{3/4}}\bigl(\sqrt[4]{82}\,-1\bigr) -\frac{2\times8\times5^{3/4}}{27}\biggr]\end{gathered}\right)} {2\bigl(\frac{27}{82^{3/4}}-\frac{8\times5^{3/4}}{27}\bigr)}.
\end{gather*}
Consequently, we arrive at
\begin{multline}\label{UL-gap-sqrt-power4}
9.00004286765564673\dotsc<\int_0^3\sqrt[4]{x^{4}+1}\,\td x+\int_{1}^{3}\sqrt[4]{x^4-1}\td x\\
<9.00004287010602764\dotsc
\end{multline}
which is neither the best nor the weakest estimate among those in Section~\ref{sec-three-examp}. The gap between the upper and lower bounds in the double inequality~\eqref{UL-gap-sqrt-power4} is $0.0000000024506\dotsc$ which, comparing with those gaps in Section~\ref{sec-three-examp}, is neither the smallest nor the biggest one.

\subsubsection{Second example}
Let
\begin{equation*}
h(x)=
\begin{cases}
e^{-1/x}, & x>0;\\
0, & x=0.
\end{cases}
\end{equation*}
Let $a=b=\frac12$ in Theorem~\ref{polya-young-thm}. Then
\begin{gather*}
h'(x)=\frac{e^{-1/x}}{x^2}, \quad h''(x)=\frac{e^{-1/x} (1-2 x)}{x^4},\\
h^{-1}\biggl(\frac12\biggr)=\frac1{\ln2}=1.44\dotsc,\quad h\biggl(\frac12\biggr)=\frac1{e^2},\\
U=h'\biggl(\frac12\biggr)=\frac4{e^2}=0.54134\dotsc, \quad L=h'\biggl(\frac1{\ln2}\biggr)=\frac{\ln^22}{2}=0.24022\dotsc,
\end{gather*}
and
\begin{gather*}
\frac{\frac{\ln^22}{2}\frac4{e^2}\bigl(\frac12-\frac1{\ln2}\bigr)^2 -2\bigl(\frac12-\frac1{\ln2}\bigr) \bigl(\frac{\ln^22}{2}\frac1{e^2}-\frac4{e^2}\frac12\bigr) +\bigl(\frac1{e^2}-\frac12\bigr)^2} {2\bigl(\frac4{e^2}-\frac{\ln^22}{2}\bigr)}\\
\le\int_0^{1/2}\frac1{e^{1/x}}\td x-\int_0^{1/2}\frac1{\ln x}\td x-\frac1{2\ln2}\\
\le-\frac{\frac{\ln^22}{2}\frac4{e^2} \bigl(\frac12-\frac1{\ln2}\bigr)^2-2\bigl(\frac12-\frac1{\ln2}\bigr) \bigl(\frac4{e^2}\frac1{e^2}-\frac{\ln^22}{2}\frac12\bigr) +\bigl(\frac1{e^2}-\frac12\bigr)^2}{2\bigl(\frac4{e^2}-\frac{\ln^22}{2}\bigr)}.
\end{gather*}
Accordingly, it follows that
\begin{multline}\label{UL-gap-exp-recip}
0.388457763460961578\dotsc<\int_0^{1/2}\frac1{e^{1/x}}\td x-\int_0^{1/2}\frac1{\ln x}\td x\\
<0.455309856619062079\dotsc
\end{multline}
whose lower bound is better, but whose upper bound is worse, than the corresponding ones in~\eqref{exp-recip-bounds}. The gap between the upper and lower bounds in the double inequality~\eqref{UL-gap-exp-recip} is $0.066852093209446\dotsc$ which is bigger than the gap $0.057414764502015\dotsc$ in the double inequality~\eqref{exp-recip-bounds}.

\subsubsection{Third example}
Let $h(x)=e^{x^2}-1$ for $x\ge0$. Then $h^{-1}(x)=\sqrt{\ln(1+x)}\,$ for $x\ge0$. Let $a=b=1$ in Theorem~\ref{polya-young-thm}. Then
\begin{gather*}
h'(x)=2xe^{x^2}, \quad h^{-1}(1)=\sqrt{\ln2}\,=0.83255\dotsc,\quad h(1)=e-1,\\
U=h'(1)=2e=5.4365\dotsc, \quad L=h'\bigl(\sqrt{\ln2}\bigr)=4\sqrt{\ln2}\,=3.3302\dotsc,
\end{gather*}
and
\begin{gather*}
\frac{8e\sqrt{\ln2}\,\bigl(1-\sqrt{\ln2}\,\bigr)^2-2\bigl(1-\sqrt{\ln2}\,\bigr)\bigl[4\sqrt{\ln2}\,(e-1)-2e\bigr] +(e-2)^2}{2\bigl(2e-4\sqrt{\ln2}\,\bigr)}\\
\le\int_0^1\bigl(e^{x^2}-1\bigr)\td x+\int_0^1\sqrt{\ln(1+x)}\,\td x-\sqrt{\ln2}\,\\
\le-\frac{8e\sqrt{\ln2}\,\bigl(1-\sqrt{\ln2}\,\bigr)^2-2\bigl(1-\sqrt{\ln2}\,\bigr) \bigl[2e(e-1)-4\sqrt{\ln2}\,\bigr] +(e-2)^2}{2\bigl(2e-4\sqrt{\ln2}\,\bigr)}.
\end{gather*}
As a result, we have
\begin{multline}\label{UL-gap-exp-squ}
2.05281277502489567\dotsc\le\int_0^1e^{x^2}\td x+\int_0^1\sqrt{\ln(1+x)}\,\td x\\
\le2.06746020503978898\dotsc
\end{multline}
whose lower bound is better, but whose upper bound is worse, than the corresponding ones in~\eqref{exp-log-sqrt-ineq}. The gap between the upper and lower bounds in the double inequality~\eqref{UL-gap-exp-squ} is $0.01464743001489\dotsc$ which is smaller than the corresponding gap $0.01578469\dotsc$ in the double inequality~\eqref{exp-log-sqrt-ineq}.

\section{More remarks}

Finally, we would like to list more remarks on our main results and possible developing directions.

\begin{rem}
Theorems~\ref{polya-young-thm} and~\ref{second-der-thm} are special cases of Theorem~\ref{further=beog-polya}. In other words, Theorems~\ref{polya-young-thm} and~\ref{second-der-thm} can be deduced from Theorem~\ref{further=beog-polya}.
\end{rem}

\begin{rem}
Some Taylor-like power expansions such as those in~\cite{Drag-Qi-Hanna-Cer-Kor-2005, Petr.tex, darboux.tex, 4180520-MS.tex} and closely related references can be used to refine Young's integral inequality~\eqref{Young-eq1}.
\end{rem}

\begin{rem}
At the present position, we conclude that many estimates of definite integrals can be used to refine Young's integral inequality~\eqref{Young-eq1}.
\end{rem}

\begin{rem}
Essentially speaking, all refinements in this paper are estimates of the area $C$ which can be geometrically demonstrated in Figures~\ref{Young-Geom-M-1.jpg} to~\ref{Young-Geom-C-2-1.jpg} and analytically expressed by~\eqref{C-area-exp-unif} in Remark~\ref{C-area-exp-unif-rem}.
\end{rem}

\end{document}